\documentclass[a4paper]{amsart}
\usepackage{amsthm,amsfonts,amsmath,amssymb}
\usepackage[abs]{overpic}
\usepackage{arydshln} 

\newtheorem{theorem}{Theorem}[section]
\newtheorem{lemma}[theorem]{Lemma}
\newtheorem{sublemma}[theorem]{Sublemma}
\newtheorem{proposition}[theorem]{Proposition}
\newtheorem{corollary}[theorem]{Corollary}
\theoremstyle{definition}
\newtheorem{definition}[theorem]{Definition}
\newtheorem{remark}[theorem]{Remark}

\newcommand{\e}{\varepsilon}
\newcommand{\de}{\delta}

\makeatletter

\@addtoreset{figure}{section}
\makeatother

\begin{document}
\title{Generalized virtualization on welded links}

\author[Haruko A. Miyazawa]{Haruko A. Miyazawa}
\address{Institute for Mathematics and Computer Science, Tsuda University,2-1-1 Tsuda-Machi, Kodaira, Tokyo, 187-8577, Japan}
\curraddr{}
\email{aida@tsuda.ac.jp}
\thanks{}

\author[Kodai Wada]{Kodai Wada}
\address{Faculty of Education and Integrated Arts and Sciences, Waseda University, 1-6-1 Nishi-Waseda, Shinjuku-ku, Tokyo, 169-8050, Japan}
\curraddr{}
\email{k.wada8@kurenai.waseda.jp}
\thanks{The second author was supported by a Grant-in-Aid for JSPS Research Fellow (\#17J08186) of the Japan Society for the Promotion of Science.}

\author[Akira Yasuhara]{Akira Yasuhara}
\address{Faculty of Commerce, Waseda University, 1-6-1 Nishi-Waseda, Shinjuku-ku, Tokyo, 169-8050, Japan}
\curraddr{}
\email{yasuhara@waseda.jp}
\thanks{The third author was partially supported by a Grant-in-Aid for Scientific Research (C) (\#17K05264) of the Japan Society for the Promotion of Science and a Waseda University Grant for Special Research Projects (\#2018S-077).}





\begin{abstract}
Let $n$ be a positive integer. 
The aim of this paper is to study two local moves $V(n)$ and $V^{n}$ on welded links,  
which are generalizations of the crossing virtualization.  
We show that the $V(n)$-move is an unknotting operation on welded knots for any $n$, and give a classification of welded links up to $V(n)$-moves. 
On the other hand, 
we give a necessary condition for which two welded links are equivalent up to $V^{n}$-moves. 
This leads to show that the $V^{n}$-move is not an unknotting operation on welded knots except for $n=1$. 
We also discuss relations among $V^{n}$-moves, associated core groups and the multiplexing of crossings.
\end{abstract}

\maketitle

\section{Introduction} 
A $\mu$-component {\em virtual link diagram} is the image of an immersion of ordered and oriented $\mu$ circles in the plane, whose transverse double points admit not only {\em classical crossings} but also {\em virtual crossings} illustrated in Figure~\ref{xing}.\footnote{For simplicity, we do not use here the usual drawing convention for virtual crossings, which is a small circle around the corresponding double point.}
We emphasize that a virtual link diagram is always ordered and oriented unless otherwise specified.

\begin{figure}[htbp]
  \begin{center}
    \begin{overpic}[width=4.5cm]{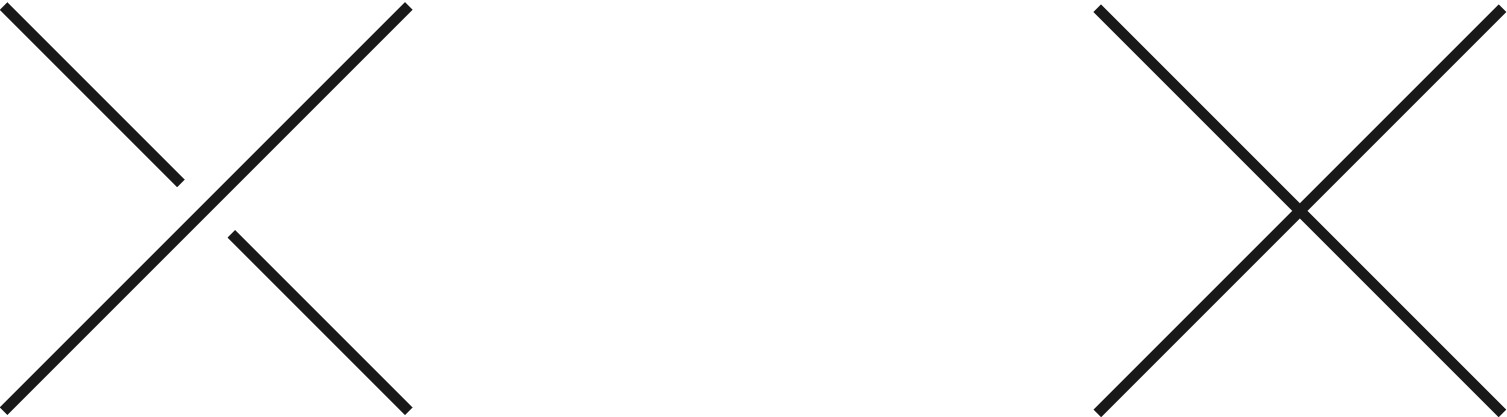}
      \put(-19,-15){classical crossing}
      \put(78,-15){virtual crossing}
    \end{overpic}
  \end{center}
  \vspace{1em}
  \caption{}
  \label{xing}
\end{figure}

A {\em virtual link} is an equivalence class of virtual link diagrams under {\em generalized Reidemeister moves}, which consist of Reidemeister moves R1--R3 and virtual moves VR1--VR4 illustrated in Figure~\ref{GRmoves}~\cite{Kauffman}.
In virtual context, there are two forbidden local moves OC and UC (meaning {\em over-crossings} and {\em under-crossings commute}, respectively) illustrated  in Figure~\ref{Fmoves}.  
The extension of the generalized Reidemeister moves which also allows the OC-move is called {\em welded Reidemeister moves}, and  
a sequence of welded Reidemeister moves is called a {\em welded isotopy}. 
A {\em welded link} is an equivalence class of virtual link diagrams 
under welded isotopy~\cite{FRR}.

\begin{figure}[htbp]
  \begin{center}
    \begin{overpic}[width=12cm]{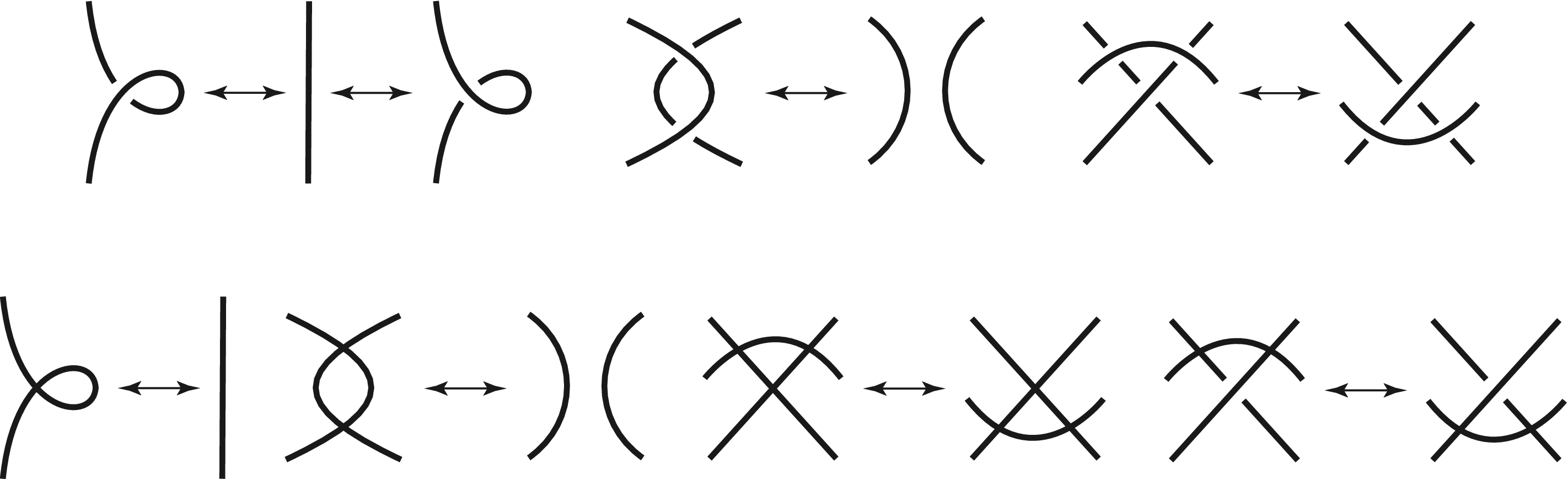}
      \put(24.5,24.5){VR1}
      \put(90.5,24.5){VR2} 
      \put(186.5,24.5){VR3}
      \put(287,24.5){VR4}
      \put(47.5,88){R1}
      \put(75,88){R1}
      \put(169.2,88){R2}
      \put(272.5,88){R3}
    \end{overpic}
  \end{center}
  \caption{Generalized Reidemeister moves}
  \label{GRmoves}
\end{figure}

\begin{figure}[htbp]
  \begin{center}
    \begin{overpic}[width=8.5cm]{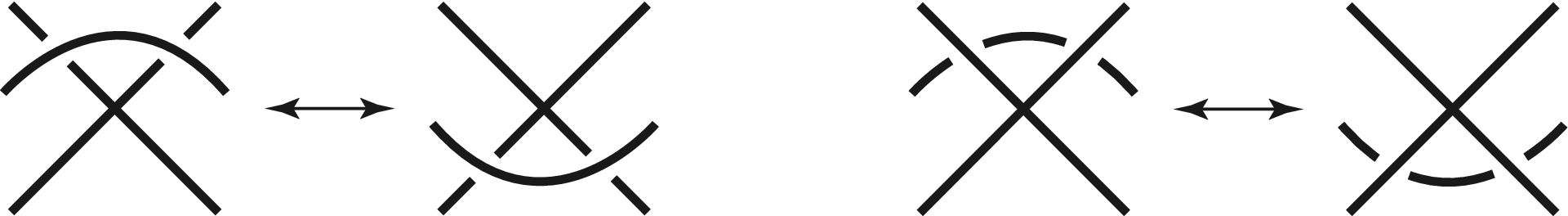}
      \put(43.4,20.5){OC}
      \put(183.5,20.5){UC}
    \end{overpic}
  \end{center}
  \caption{Forbidden moves OC and UC}
  \label{Fmoves}
\end{figure}

A virtual link diagram is {\em classical} if it has no virtual crossings, 
and a welded link is {\em classical} if it has a classical link diagram.
M.~Goussarov, M.~Polyak and O.~Viro~\cite{GPV} essentially proved that welded isotopic classical link diagrams can be related by Reidemeister moves R1--R3. 
Therefore, welded links can be viewed as a natural extension of classical links.
We remark that any virtual knot diagram can be unknotted by UC-moves and welded Reidemeister moves~\cite{GPV,Kanenobu,N}.  
This result is one reason why the UC-move is still forbidden in welded context.

In classical knot theory, 
local moves have played important roles and hence have been studied widely; see for example~\cite{M,MN,HNT,A,MiyaY,DP02,DP04}. 
Recently, some ``classical'' 
local moves, which exchange classical tangle diagrams,  
 have been studied for welded knots and links~\cite{ABMW-JMSJ,ABMW,S,NNSY}.  
In this paper, we will study ``non-classical'' local moves for welded links. 
A typical non-classical local move is the crossing virtualization. 
The {\em crossing virtualization} V is a local move on a virtual link diagram replacing a classical crossing with a virtual one; see the left-hand side of Figure~\ref{virtualization}. 
We remark that 
any virtual link diagram can be deformed into a diagram of the trivial link by applying the crossing virtualization repeatedly. 
The crossing virtualization is equivalent to the local move illustrated  in the right-hand side of Figure~\ref{virtualization}. 
Here, two local moves are {\em equivalent} if each move is realized by a sequence of the other move and welded Reidemeister moves.

\begin{figure}[htbp]
  \begin{center}
    \begin{overpic}[width=12cm]{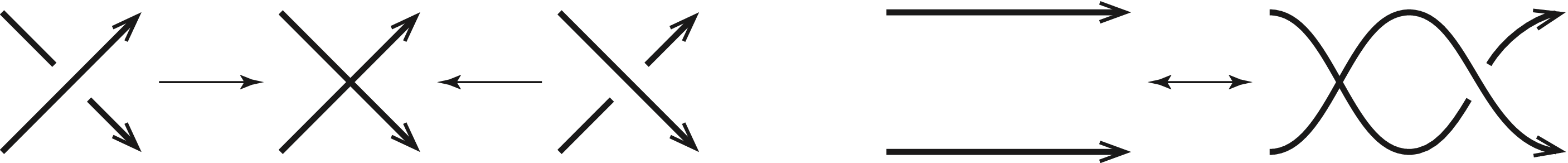}
      \put(41,20){V}
      \put(104,20){V}
    \end{overpic}
  \end{center}
  \caption{Crossing virtualization}
  \label{virtualization}
\end{figure}

Let $n$ be a positive integer. 
The aim of this paper is to study 
two oriented local moves $V(n)$ and $V^{n}$ 
illustrated  in the upper and lower sides of Figure~\ref{GV}, respectively. 
They are considered as generalizations of the crossing virtualization. 
In fact, both $V(1)$- and $V^{1}$-moves are equivalent to the crossing virtualization. 
Note that if $n$ is even then the $V(n)$-move may change the number of components. 
Two welded links are {\em $V(n)$-equivalent} (resp. {\em $V^{n}$-equivalent}) if their diagrams are related by $V(n)$-moves (resp. $V^{n}$-moves) and welded Reidemeister moves.

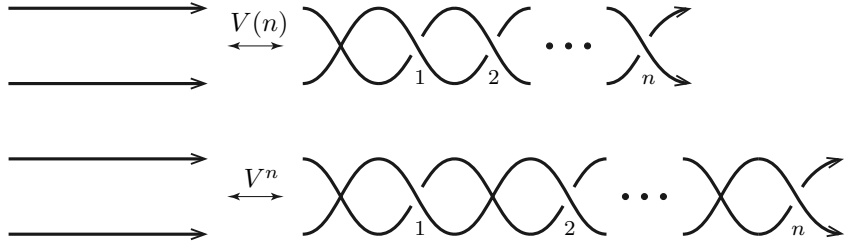
\begin{figure}[htbp]
  \begin{center}
    \begin{overpic}[width=11cm]{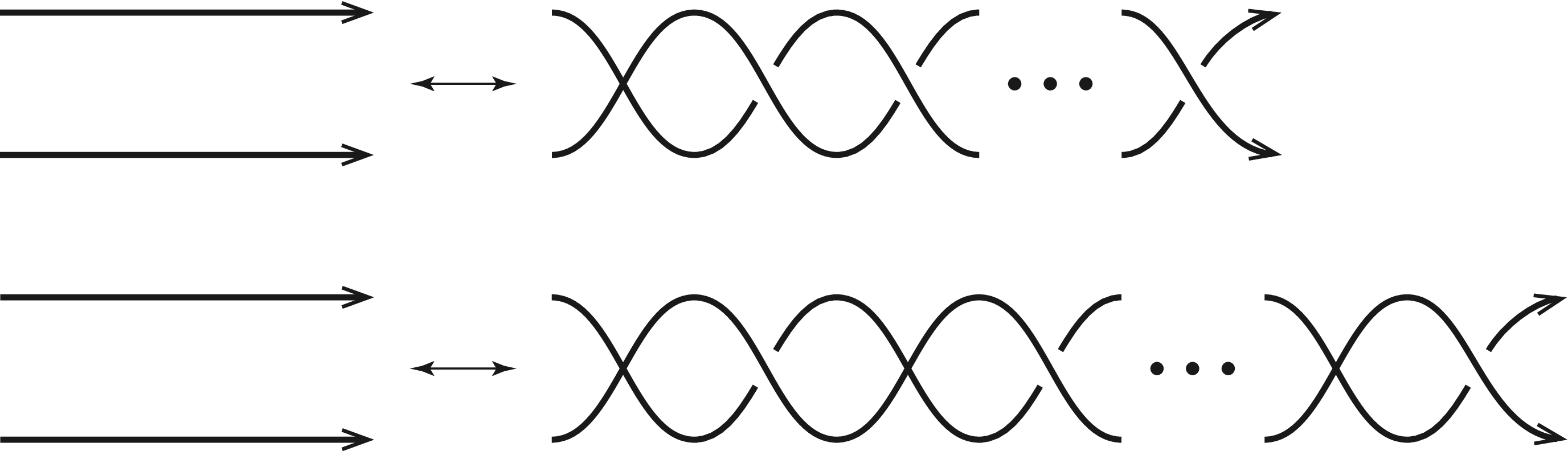}
      \put(83,79){$V(n)$}
      \put(152,59){{\footnotesize $1$}}
      \put(179.5,59){{\footnotesize $2$}}
      \put(237,59){{\footnotesize $n$}}
      \put(88,20){$V^{n}$}
      \put(152,2){{\footnotesize $1$}}
      \put(208,2){{\footnotesize $2$}}
      \put(293,2){{\footnotesize $n$}}
    \end{overpic}
  \end{center}
  \caption{$V(n)$- and $V^{n}$-moves}
  \label{GV}
\end{figure}

We have that the $V(n)$-move is an unknotting operation on welded knots for any positive integer $n$ because a UC-move is realized by a sequence of the $V(n)$-move and welded Reidemeister moves (Proposition~\ref{prop-UC}).
Moreover, we give a classification of welded links up to $V(n)$-equivalence in the sense of Theorems~\ref{th-even} and~\ref{th-odd}.

\begin{theorem}
\label{th-even}
Let $n$ be an even number. 
Any welded link is $V(n)$-equivalent to the unknot.
\end{theorem}

Let $D$ be a virtual link diagram.   
For any $i,j$ $(i\neq j)$, let $\lambda_{ij}(D)$ denote
the sum of the signs of all classical crossings of $D$ whose overpasses and underpasses belong to the $i$th and $j$th component, respectively. 
The integer $\lambda_{ij}(D)$ is a welded link invariant and also preserved by UC-moves. 
For a welded link $L$, 
the {\em ordered linking number $\lambda_{ij}(L)$} is defined to be $\lambda_{ij}(D)$ for a diagram $D$ of $L$.

It is not hard to see that if $n$ is odd then the modulo-$n$ reduction of $\lambda_{ij}(L)+\lambda_{ji}(L)$ is preserved by $V(n)$-moves. 
Using these invariants we have the following.

\begin{theorem}
\label{th-odd}
Let $n$ be an odd number. 
Two $\mu$-component welded links 
$L$ and $L'$ are $V(n)$-equivalent if and only if
$\lambda_{ij}(L)+\lambda_{ji}(L)\equiv \lambda_{ij}(L')+\lambda_{ji}(L')$ $\pmod{n}$ for any $i,j$ $(1\leq i<j\leq\mu)$.
\end{theorem}

On the other hand, $V^{n}$-moves preserve the modulo-$n$ reduction of $\lambda_{ij}(L)$ for any positive integer $n$. 
However, 
these invariants are not strong enough to classify welded links up to $V^{n}$-equivalence 
because the $V^{n}$-move is not an unknotting operation on welded knots except for $n=1$ (Proposition~\ref{prop-trefoil}).
Considering the UC-move, which is an unknotting operation for welded knots, we have the following.

\begin{theorem}
\label{th-V^n}
Let $n$ be a positive integer. 
Two $\mu$-component welded links 
$L$ and $L'$ are $(V^{n}+{\rm UC})$-equivalent if and only if  
$\lambda_{ij}(L)\equiv\lambda_{ij}(L')$ $\pmod{n}$ for any $i,j$ $(1\leq i\neq j\leq\mu)$. 
\end{theorem}

\noindent
Here, two welded links are {\em $(V^{n}+{\rm UC})$-equivalent} 
if their diagrams are related by $V^{n}$-moves, UC-moves and welded Reidemeister moves.

\begin{remark}
Theorem~\ref{th-V^n} easily follows from the classification of welded links up to UC-moves, given in \cite[Theorem~8]{O} and \cite[Theorem~4.7]{Nasybullov}. 
In this paper, we will prove Theorem~\ref{th-V^n} without using their classification. 
Our proof of the theorem contains a simple proof for their classification. 
\end{remark}

We also discuss relations among {\em unoriented} $V^{n}$-moves, associated core groups and the multiplexing of crossings. 
The {\em associated core group} is known as an unoriented classical link invariant~\cite{J,Kelly,FR,W}. 
This group is naturally extended to an unoriented welded link invariant, and furthermore, it is preserved by unoriented $V^{n}$-moves for any even number $n$ (Proposition~\ref{prop-V2}). 
In~\cite{MWY17} the authors introduced the notion of {\em multiplexing} of crossings for an unoriented $\mu$-component welded link $L$, 
which yields a new unoriented welded link $L(m_{1},\ldots,m_{\mu})$ associated with a $\mu$-tuple $(m_{1},\ldots,m_{\mu})$ of 
integers.  
For any $\mu$-tuple $(m_{1},\ldots,m_{\mu})$ of
even numbers, 
$L(m_{1},\ldots,m_{\mu})$ is deformed into the $\mu$-component trivial link by unoriented $V^{2}$-moves (Proposition~\ref{th-V2}). 
As a consequence, 
we have that there are infinitely many nontrivial welded knots whose associated core groups are isomorphic to that of the unknot (Theorem~\ref{prop-infinitely}).

\section{Arrow calculus}\label{sec-arrow}
To show Theorems~\ref{th-even}, \ref{th-odd} and \ref{th-V^n}, we will use {\em arrow calculus} introduced by J.-B. Meilhan and the third author in~\cite{MY}. 
In this section, we will briefly recall the basic definitions of arrow calculus from~\cite{MY}. 
We only need the notion of {\rm w}-arrow,  
and refer the reader to~\cite{MY} for more details of arrow calculus.

\begin{definition}
Let $D$ be a virtual link diagram. 
A {\em {\rm w}-arrow $\gamma$} for $D$ is an oriented arc immersed in the plane of the diagram such that;
\begin{enumerate}
\item the endpoints of $\gamma$ are contained in $D\setminus\{\text{crossings of $D$}\}$, 

\item all singularities of $\gamma$ are virtual crossings, 

\item all singularities between $D$ and $\gamma$ are virtual crossings, and

\item $\gamma$ has a number (possibly zero) of decorations $\bullet$ 
on the interior of $\gamma$, called {\em twists}, which are disjoint from all crossings.  
\end{enumerate}

The initial and terminal points of $\gamma$ are called the {\em tail} and the {\em head}, respectively. 
For a union of {\rm w}-arrows for $D$, all crossings among {\rm w}-arrows are assumed to be virtual. 
\end{definition}

Hereafter, diagrams are drawn with bold lines while {\rm w}-arrows are drawn with thin lines.

Let $A$ be a union of {\rm w}-arrows for $D$. 
Now we define {\em surgery along $A$} on $D$ which yields a new virtual link diagram, denoted by $D_{A}$, as follows. 
Suppose that there exists a disk in the plane which intersects $D\cup A$ as illustrated  in Figure~\ref{surgery}. 
Then the figure indicates the result of surgery along a {\rm w}-arrow of $A$ on $D$. 
We emphasize that the surgery move depends on the orientation of the strand of $D$ containing the tail of the {\rm w}-arrow.

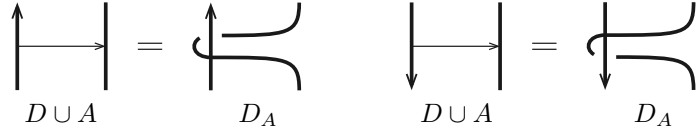
\begin{figure}[htbp]
  \begin{center}
    \begin{overpic}[width=9cm]{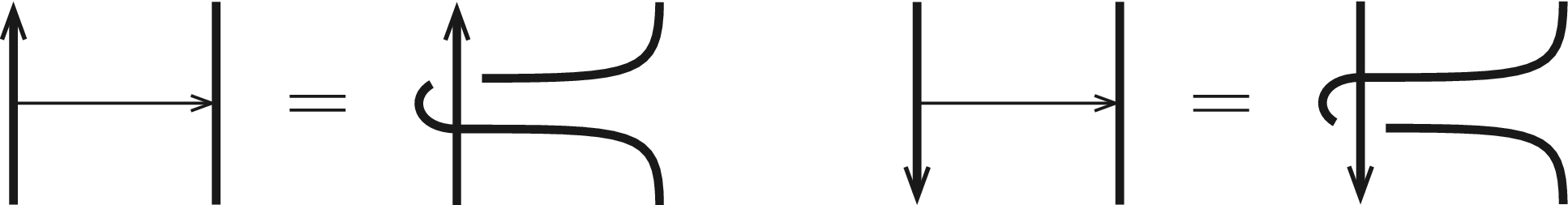}
      \put(5,-12){$D\cup A$} 
      \put(85,-12){$D_{A}$} 
      \put(153,-12){$D\cup A$}
      \put(233,-12){$D_{A}$} 
    \end{overpic}
  \end{center}
  \vspace{1em}
  \caption{Surgery along a {\rm w}-arrow of $A$ on $D$}
  \label{surgery}
\end{figure}

If a {\rm w}-arrow of $A$ intersects a (possibly the same) {\rm w}-arrow (resp. $D$), then the result of surgery is essentially the same as above but each intersection introduces virtual crossings illustrated  in the left-hand side (resp. center) of Figure~\ref{surgery2}.
Moreover, if a {\rm w}-arrow of $A$ has some twists, then each twist is converted to a half-twist whose crossing is virtual; see  the right-hand side of Figure~\ref{surgery2}.

\begin{figure}[htbp]
  \begin{center}
    \begin{overpic}[width=12cm]{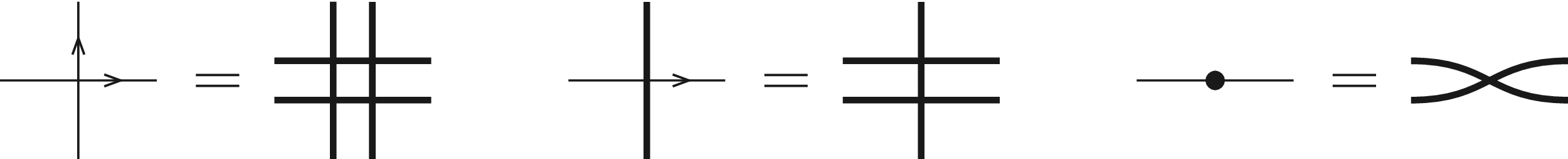}
      \put(-5,21){$A$}
      \put(12,-13){$A$} 
      \put(70,-13){$D_{A}$} 
      \put(119,21){$A$}
      \put(136,-13){$D$}
      \put(193,-13){$D_{A}$} 
      \put(261,-13){$A$}
      \put(318,-13){$D_{A}$} 
    \end{overpic}
  \end{center}
  \vspace{1em}
  \caption{}
  \label{surgery2}
\end{figure}

An {\em arrow presentation} for a virtual link diagram $D$ is a pair $(T,A)$ of a virtual link diagram $T$ without classical crossings and a union $A$ of {\rm w}-arrows for $T$ such that $T_{A}$ is welded isotopic to $D$. 
Every virtual link diagram has an arrow presentation 
because any classical crossing can be replaced by a virtual one with a {\rm w}-arrow; see Figure~\ref{Aprst}.

\begin{figure}[htbp]
  \begin{center}
    \begin{overpic}[width=6cm]{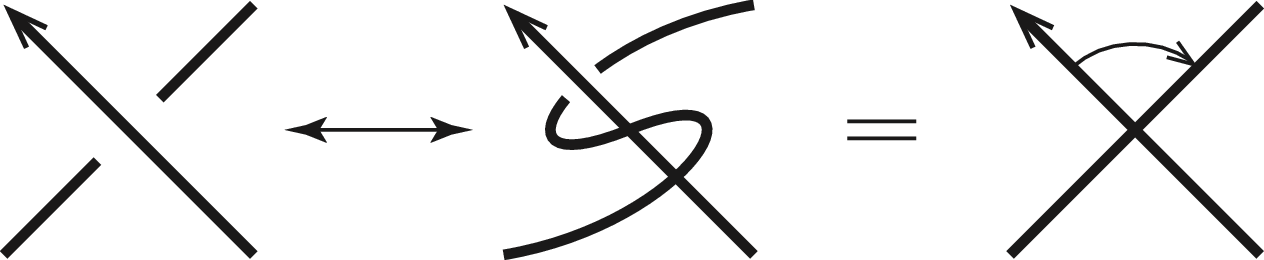}
      \put(41,22){VR2}
    \end{overpic}
  \end{center}
    \caption{}
  \label{Aprst}
\end{figure}

Two arrow presentations are {\em equivalent} if the surgeries yield welded isotopic virtual link diagrams. 
{\em Arrow moves} consist of virtual moves VR1--VR3 involving {\rm w}-arrows and/or strands of $D$ and the local moves AR1--AR10 on arrow presentations  illustrated  in Figure~\ref{Amoves}. 
Here, each vertical strand in AR1--AR3 is either a strand of $D$ or a {\rm w}-arrow, 
and the symbol $\circ$ on a {\rm w}-arrow in AR8 and AR10 denotes that the {\rm w}-arrow may or may not contain a twist. 
Two arrow presentations are equivalent if and only if they are related by arrow moves~\cite[Theorem 4.5]{MY}.

\begin{figure}[htbp]
  \begin{center}
    \begin{overpic}[width=11.5cm]{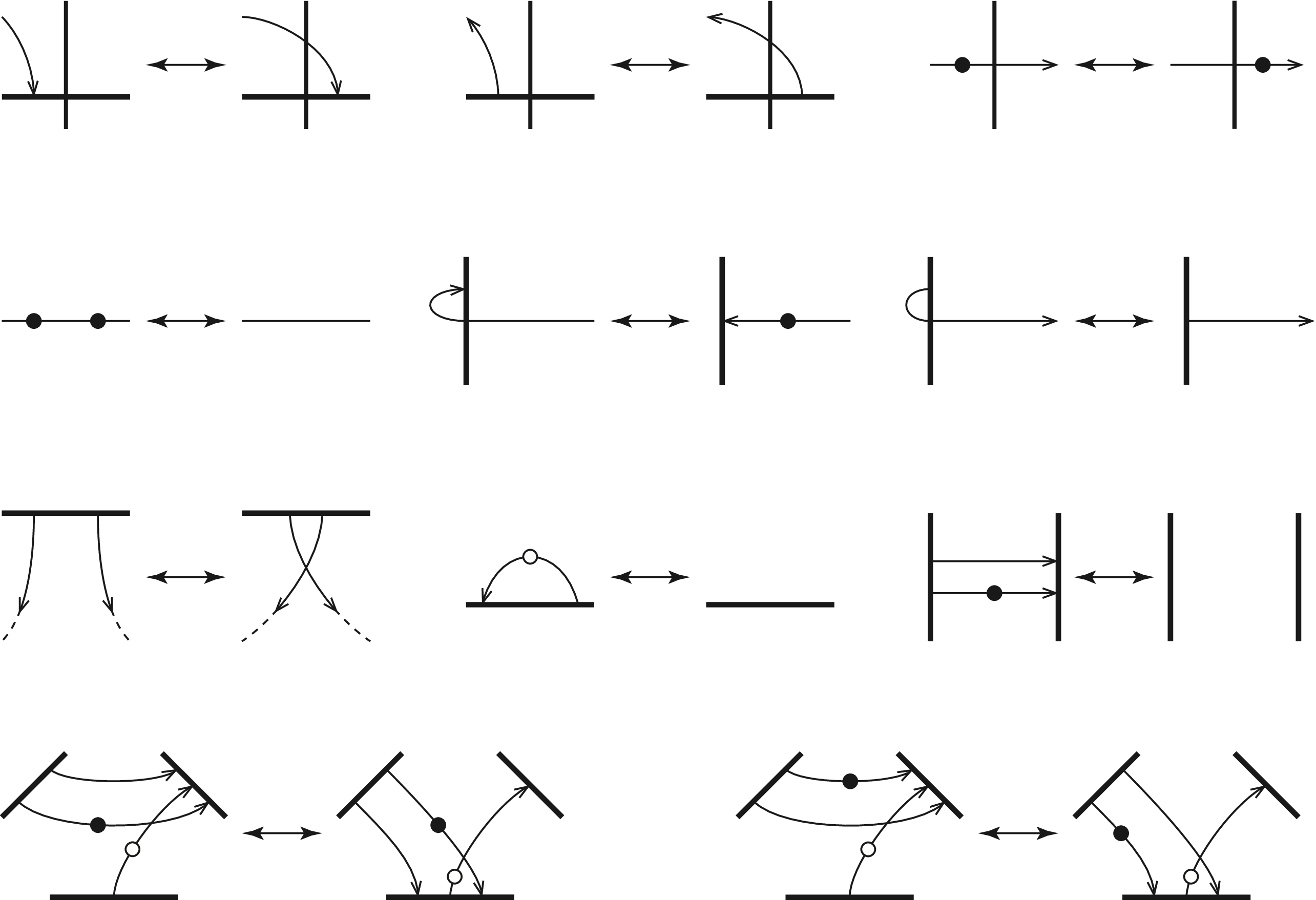}
      \put(37,212){AR1} 
      \put(152,212){AR2} 
      \put(267,212){AR3}
      \put(37,148.5){AR4} 
      \put(152,148.5){AR5}
      \put(267,148.5){AR6} 
      \put(37,84.5){AR7} 
      \put(152,84.5){AR8} 
      \put(267,84.5){AR9}
      \put(58,21){AR10} 
      \put(241.5,21){AR10}
    \end{overpic}
  \end{center}
  \caption{Arrow moves AR1--AR10}
  \label{Amoves}
\end{figure}
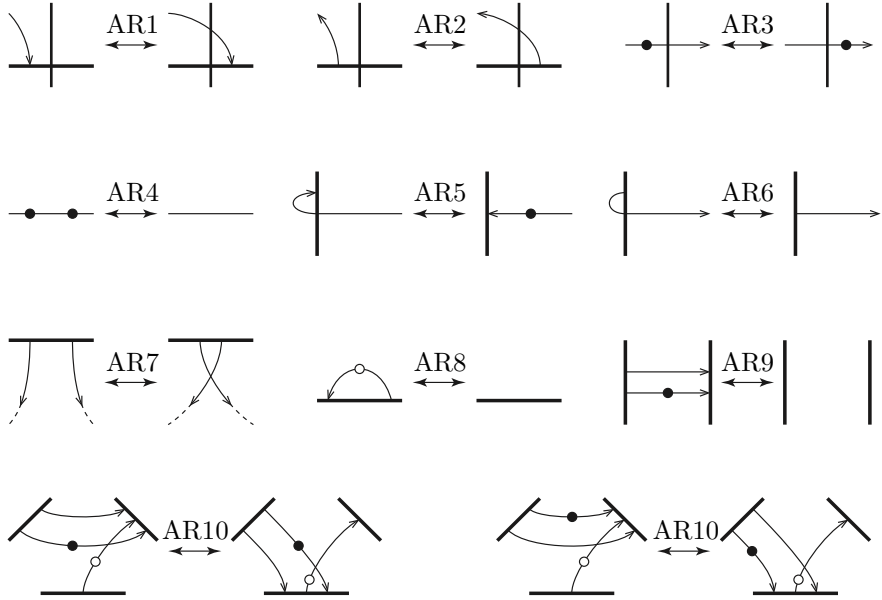

In the rest of this section, we will introduce several local moves on arrow presentations. 
We first consider two {\em allowable moves} AR11 and AR12 as illustrated in Figures~\ref{AR11} and \ref{AR12}, respectively.  
Each of the moves is realized by a sequence of arrow moves. 
Figure~\ref{pf-AR112} shows that the left-hand side moves in Figures~\ref{AR11} and \ref{AR12} are realized by arrow moves, 
where $\overset{\text{AR}}{\sim}$ in the figure denotes a sequence of arrow moves.
The other cases are similarly shown.

\begin{figure}[htbp]
  \begin{center}
    \begin{overpic}[width=11cm]{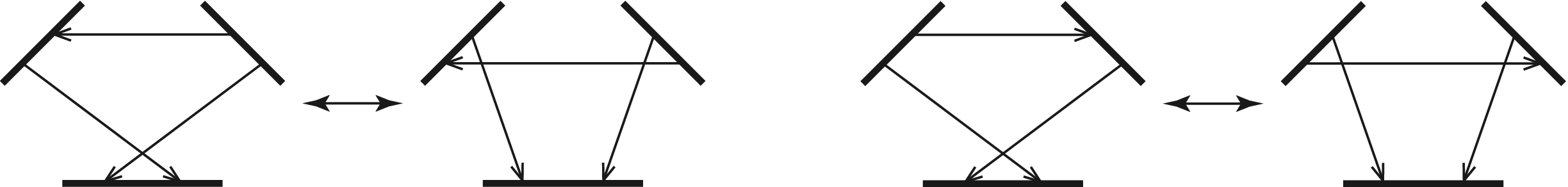}
      \put(58,20.5){AR11} 
      \put(230,20.5){AR11}
    \end{overpic}
  \end{center}
  \caption{Allowable move AR11}
  \label{AR11}

\vspace{1em}

  \begin{center}
    \begin{overpic}[width=11cm]{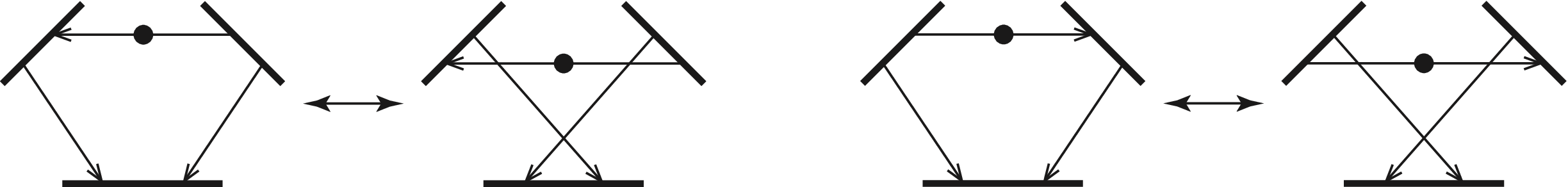}
      \put(58,20.5){AR12} 
      \put(230,20.5){AR12}
    \end{overpic}
  \end{center}
  \caption{Allowable move AR12}
  \label{AR12}
\end{figure}

\begin{figure}[htbp]
  \begin{flushright}
    \begin{overpic}[width=11.5cm]{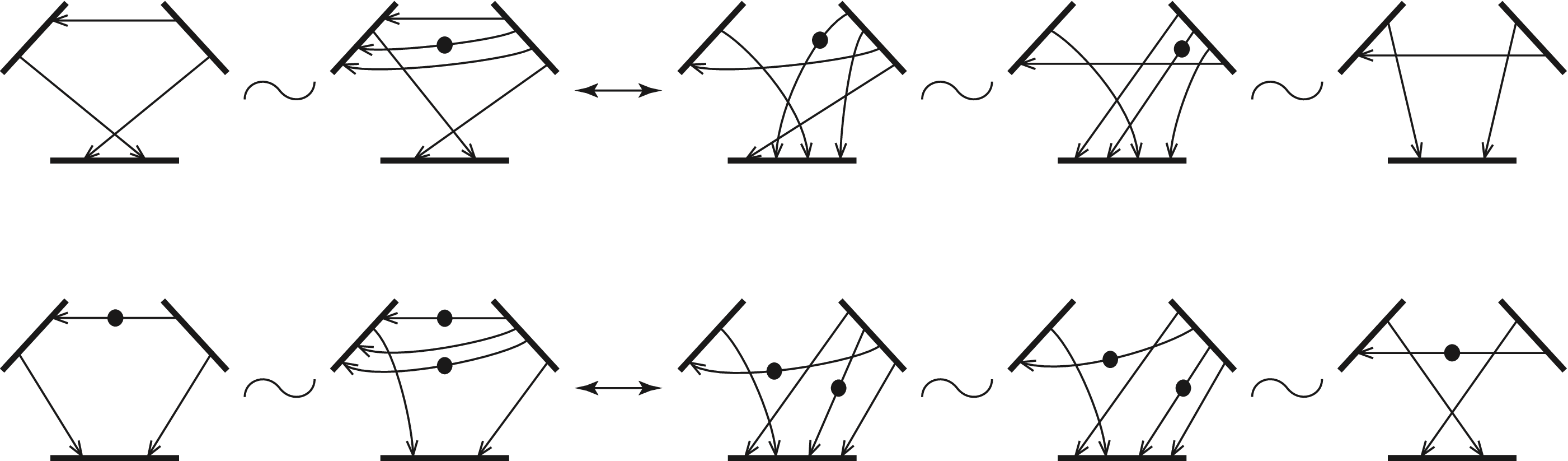}
      \put(-30,75){{\small AR11 :}}
      \put(-30,12){{\small AR12 :}}
      \put(51,81){{\small AR}} 
      \put(118,81){{\small AR10}} 
      \put(192.5,81){{\small AR}}
      \put(261.5,81){{\small AR}}
      
      \put(51,19){{\small AR}} 
      \put(118,19){{\small AR10}} 
      \put(192.5,19){{\small AR}} 
      \put(261.5,19){{\small AR}} 
    \end{overpic}
  \end{flushright}
  \caption{}
  \label{pf-AR112}
\end{figure}

The {\em heads exchange move}\footnote{Note that our definition slight differs from the one given in~\cite[Lemma 5.14]{MY}.} 
is a local move on an arrow presentation exchanging positions of consecutive two heads of {\rm w}-arrows; see Figure~\ref{Hexch}. 
While there are several kinds of heads exchange moves 
depending on the orientation of the strand containing the tail and existence or nonexistence of a twist for a {\rm w}-arrow, 
we have the following.

\begin{figure}[htbp]
  \begin{center}
    \begin{overpic}[width=5cm]{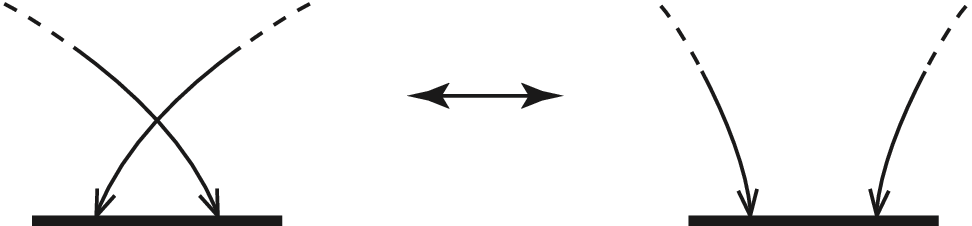}
    \end{overpic}
  \end{center}
  \caption{Heads exchange move}
  \label{Hexch}
\end{figure}

\begin{sublemma}
\label{lem-Hexch}
A heads exchange move is realized by a sequence of the {\rm H}-move illustrated  in Figure~$\ref{H}$ and arrow moves.
\end{sublemma}

\begin{figure}[htbp]
  \begin{center}
    \begin{overpic}[width=5cm]{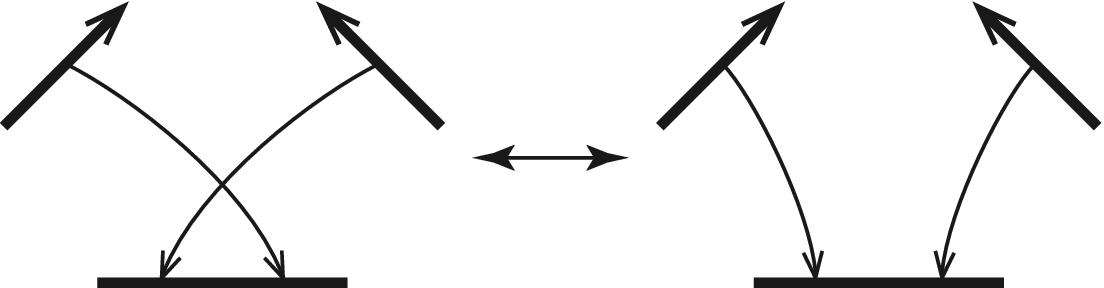}
      \put(67,20){H}
    \end{overpic}
  \end{center}
  \caption{H-move}
  \label{H}
\end{figure}

\begin{proof}
We demonstrate that two of the heads exchange moves are realized by  sequences of the H-move and arrow moves. 
The upper side of Figure~\ref{pf-Hexch} indicates the proof 
when the orientation of the strand containing the tail of a single {\rm w}-arrow is opposite to that of the H-move. 
The lower side of Figure~\ref{pf-Hexch} indicates the proof for the case where a {\rm w}-arrow has a twist. 
It is not hard to show the other cases. 
\end{proof}

\begin{figure}[htbp]
  \begin{center}
    \begin{overpic}[width=10.5cm]{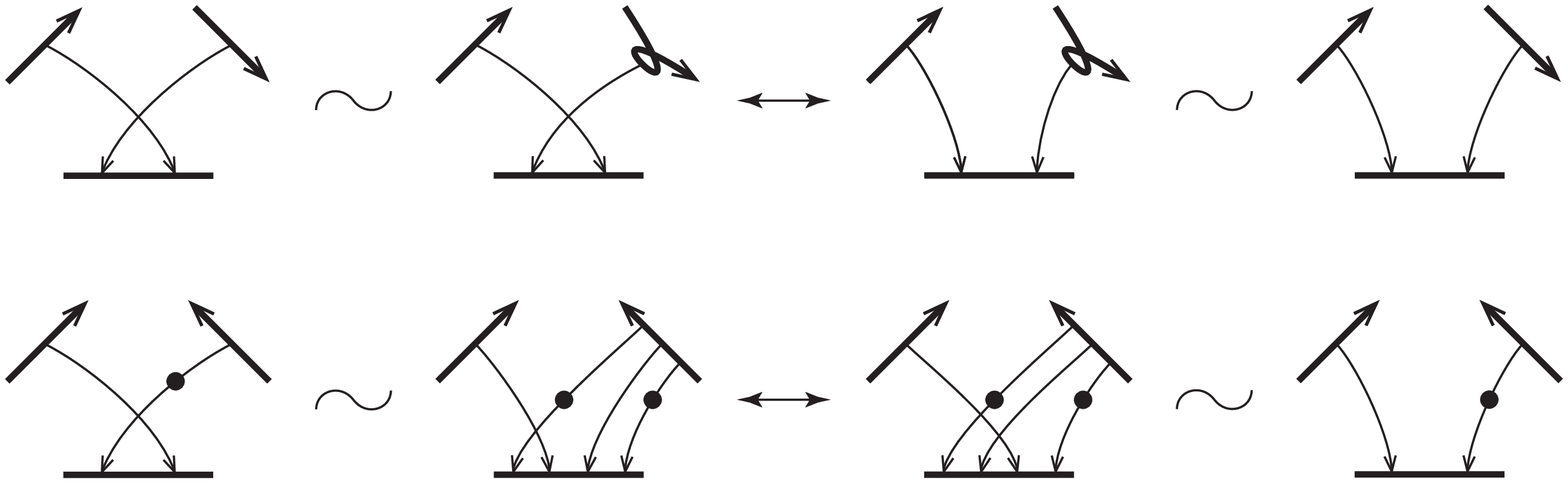}
      \put(59.5,77){AR}
      \put(145.5,76){H}
      \put(224.5,77){AR}
      \put(59.5,19.5){AR}
      \put(145,19){H}
      \put(224.5,19.5){AR}
    \end{overpic}
  \end{center}
  \caption{}
  \label{pf-Hexch}
\end{figure}

The {\em head-tail exchange move}\footnote{Note that our definition slight differs from the one given in~\cite[Lemma 5.16]{MY}.} 
is a local move on an arrow presentation exchanging positions of consecutive a head and a tail of {\rm w}-arrows; see Figure~\ref{HTexch}.

\begin{figure}[htbp]
  \begin{center}
    \begin{overpic}[width=5cm]{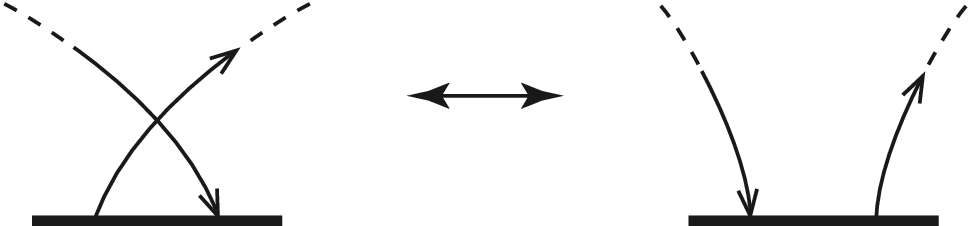}
    \end{overpic}
  \end{center}
  \caption{Head-tail exchange move}
  \label{HTexch}
\end{figure}

\begin{sublemma}
\label{lem-HTexch}
A head-tail exchange move is realized by a sequence of the heads exchange move and arrow moves. 
\end{sublemma}

\begin{proof}
See Figure~\ref{pf-HTexch}. 
\end{proof}

\begin{figure}[htbp]
  \begin{center}
    \begin{overpic}[width=12.5cm]{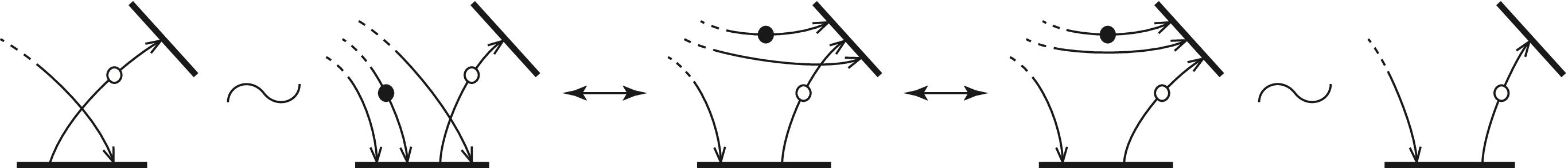}
      \put(52,22){AR}
      \put(125,22){AR10}
      \put(204,22){{\footnotesize heads}}
      \put(198,8){{\footnotesize exchange}}
      \put(286,22){AR}
    \end{overpic}
  \end{center}
  \caption{}
  \label{pf-HTexch}
\end{figure}

Three kinds of moves, AR7, heads exchange and head-tail exchange moves, are called {\em ends exchange moves}. 
From Sublemmas~\ref{lem-Hexch} and \ref{lem-HTexch} we have the following.

\begin{lemma}
\label{lem-ends} 
An ends exchange move is realized by a sequence of the {\rm H}-move and arrow moves. 
\end{lemma}

\section{$V(n)$-moves and UC-moves}
\label{sec-moves}
In this section, we will show that the $V(n)$-move is an unknotting operation on welded knots. 
We start with the following lemma concerning the UC-move.

\begin{lemma}
\label{lem-UC}
An arrow presentation for a {\rm UC}-move is realized by a sequence of the heads exchange move and arrow moves. 
Conversely, surgery along a heads exchange move is realized by a sequence of the {\rm UC}-move and welded Reidemeister moves. 
\end{lemma}

\begin{proof}
Figure~\ref{pf-UC} shows that an arrow presentation for a UC-move is realized by a sequence of the heads exchange move and arrow moves. 
In the figure, we choose certain orientations of two strands at the virtual crossing. 
The other cases are similarly shown.

\begin{figure}[htbp]
  \begin{center}
    \begin{overpic}[width=10cm]{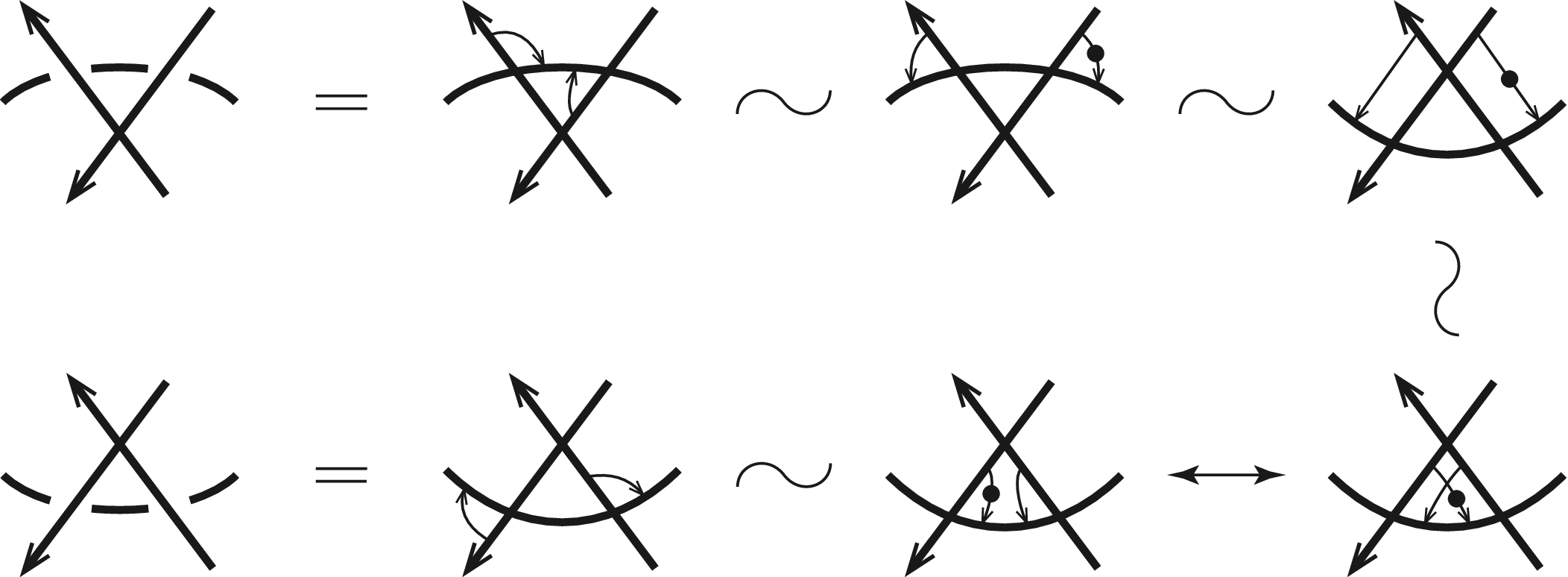}
      \put(134,91){AR}
      \put(214.5,91){AR}
      \put(267,49){AR}
      \put(212,23){{\footnotesize heads}}
      \put(206,9){{\footnotesize exchange}}
      \put(134,23){AR}
    \end{overpic}
  \end{center}
  \caption{} 
  \label{pf-UC}
\end{figure}

Conversely, Figure~\ref{UC-Hexch} shows that surgery along an H-move is realized by a sequence of the UC-move and welded Reidemeister moves,  
where $\overset{\text{w}}{\sim}$ in the figure denotes a welded isotopy. 
This and Sublemma~\ref{lem-Hexch} complete the proof. 
\end{proof}

\begin{figure}[htbp]
  \begin{center}
    \begin{overpic}[width=10cm]{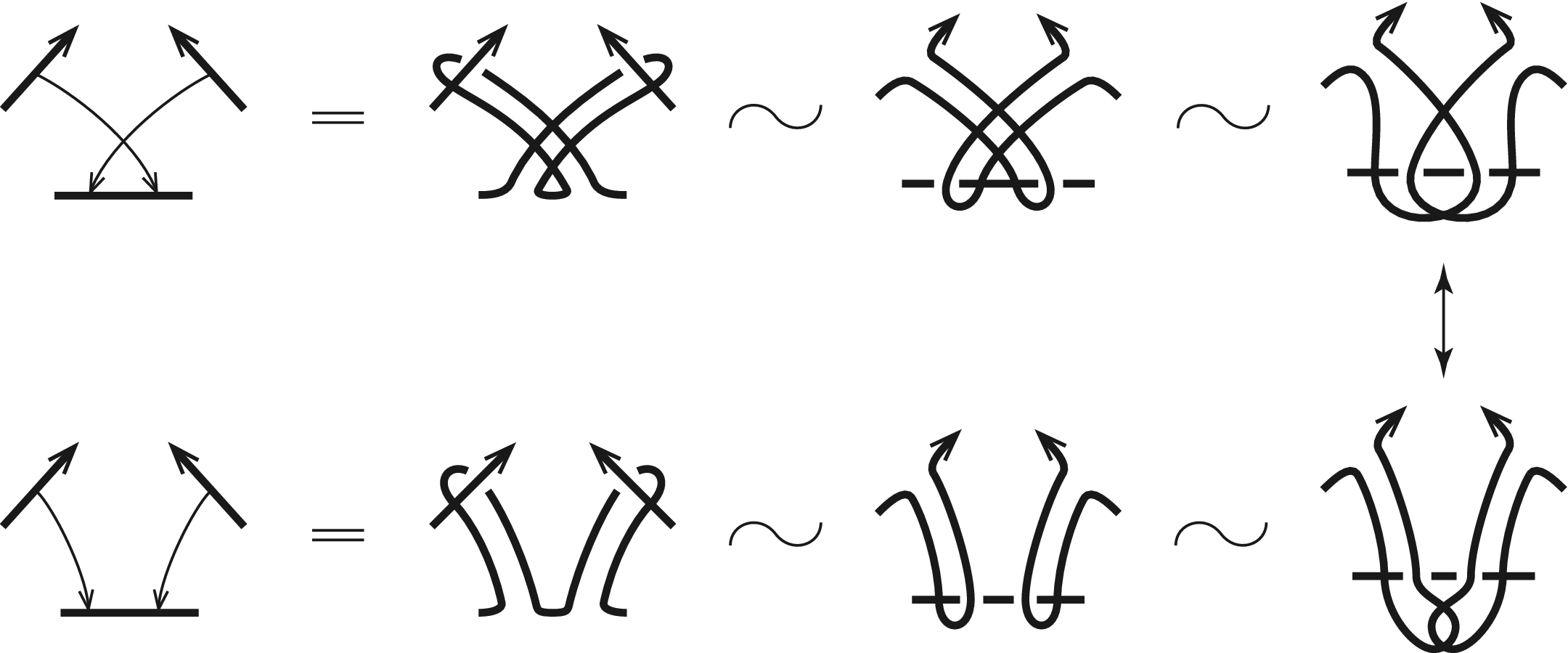}
      \put(137,101){w}
      \put(219,101){w}
      \put(137,25.5){w}
      \put(218,25.5){w}
      \put(265,57){UC}
    \end{overpic}
  \end{center}
  \caption{} 
  \label{UC-Hexch}
\end{figure}

We define the {\em $A(n)$-move} as a local move on an arrow presentation depending on the parity of $n$. 
The $A(n)$-move is illustrated  in Figure~\ref{A(n)odd} (resp. Figure~\ref{A(n)even}) when $n$ is odd (resp. even).

\begin{figure}[htbp]
  \begin{center}
    \begin{overpic}[width=10cm]{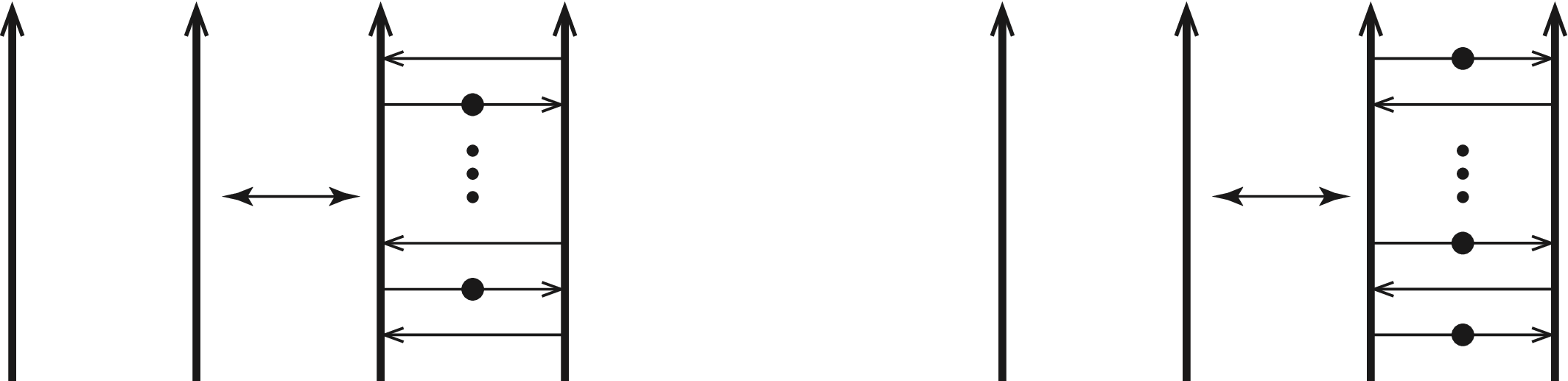}
      \put(44,39){$A(n)$}
      \put(106,5){{\footnotesize $1$}}
      \put(106,14){{\footnotesize $2$}}
      \put(106,57){{\footnotesize $n$}}
      
      \put(222,39){$A(n)$}
      \put(285.5,5){{\footnotesize $1$}}
      \put(285.5,14){{\footnotesize $2$}}
      \put(285.5,57){{\footnotesize $n$}}
    \end{overpic}
  \end{center}
  \caption{$A(n)$-move on an arrow presentation when $n$ is odd}
  \label{A(n)odd}
\end{figure}

\begin{figure}[htbp]
  \begin{center}
    \begin{overpic}[width=10cm]{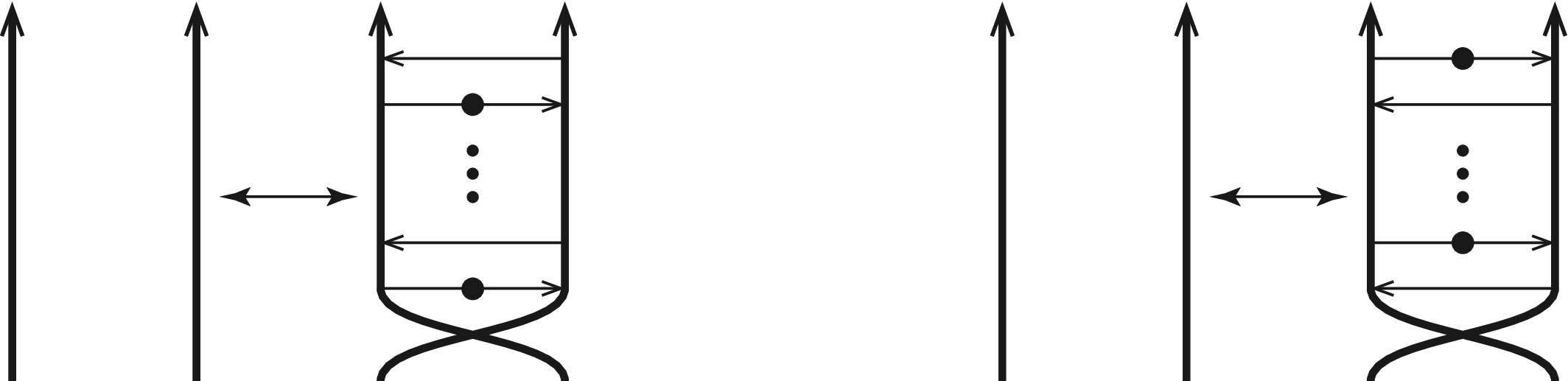}
      \put(44,39){$A(n)$}
      \put(106,14){{\footnotesize $1$}}
      \put(106,23){{\footnotesize $2$}}
      \put(106,57){{\footnotesize $n$}}
      
      \put(222,39){$A(n)$}
      \put(285.5,14){{\footnotesize $1$}}
      \put(285.5,23){{\footnotesize $2$}}
      \put(285.5,57){{\footnotesize $n$}}
    \end{overpic}
  \end{center}
  \caption{$A(n)$-move on an arrow presentation when $n$ is even}
  \label{A(n)even}
\end{figure}
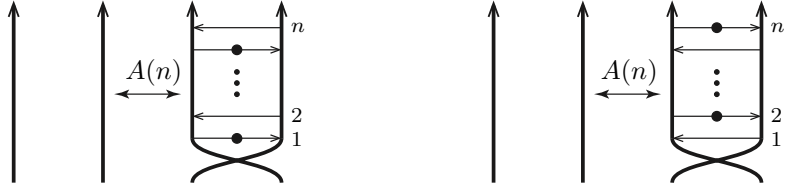

\begin{lemma}
\label{lem-HvsA(n)}
Let $n$ be a positive integer. 
An ends exchange move is realized by a sequence of the $A(n)$-move and arrow moves. 
\end{lemma}

\begin{proof}
By Lemma~\ref{lem-ends}, 
it suffices to show that an H-move is realized by a sequence of the $A(n)$-move and arrow moves for any $n$. 
The upper (resp. lower) side of Figure~\ref{pf-propUC} indicates the proof for the case where $n$ is odd (resp. even),  
while the figure describes only when $n=3$ (resp. $n=2$).
\end{proof}

\begin{figure}[htbp]
  \begin{center}
    \begin{overpic}[width=12cm]{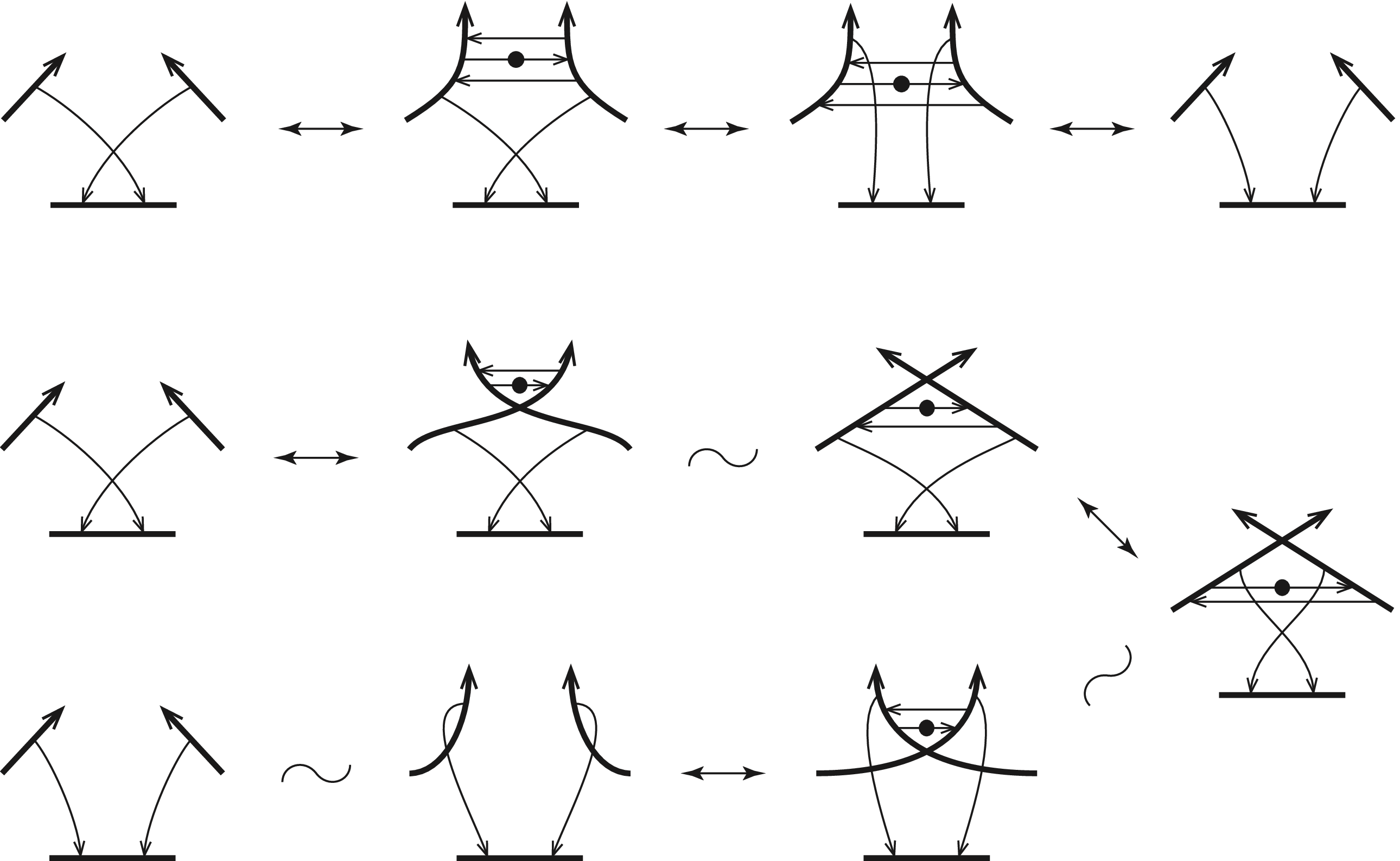}
      \put(68,184){$A(n)$}
      \put(154,184){AR11,12}
      \put(256,184){$A(n)$}
      \put(67,104){$A(n)$}
      \put(168,103){AR}
      \put(266,89){AR11,12}
      \put(271,36){AR}
      \put(166,27){$A(n)$}
      \put(69.5,26){AR}
    \end{overpic}
  \end{center}
  \caption{}
  \label{pf-propUC}
\end{figure}

\begin{lemma}\label{lem-A(n)V(n)}
An arrow presentation for a $V(n)$-move is realized by a sequence of the $A(n)$-move and arrow moves. 
Conversely, surgery along an $A(n)$-move is realized by a sequence of the $V(n)$-move and welded Reidemeister moves. 
\end{lemma}

\begin{proof}
It is not hard to see that the right-hand side move in Figure~\ref{A(n)odd} (resp. Figure~\ref{A(n)even}) is realized by a sequence of the left-hand side move in Figure~\ref{A(n)odd} (resp. Figure~\ref{A(n)even}) and arrow moves.  
See, for example, Figure~\ref{pf-A(n)odd} in the case when $n$ is odd. 
Furthermore, 
an arrow presentation for a $V(n)$-move is realized by a sequence of the left-hand side move in Figure~\ref{A(n)odd} (resp. Figure~\ref{A(n)even}) and arrow moves when $n$ is odd (resp. even). 
Conversely, it is obvious that surgery along the left-hand side move in Figure~\ref{A(n)odd} (resp. Figure~\ref{A(n)even}) is realized by a sequence of the $V(n)$-move and welded Reidemeister moves when $n$ is odd (resp. even).
Therefore, we have the conclusion. 
\end{proof}

\begin{figure}[htbp]
  \begin{center}
    \begin{overpic}[width=9cm]{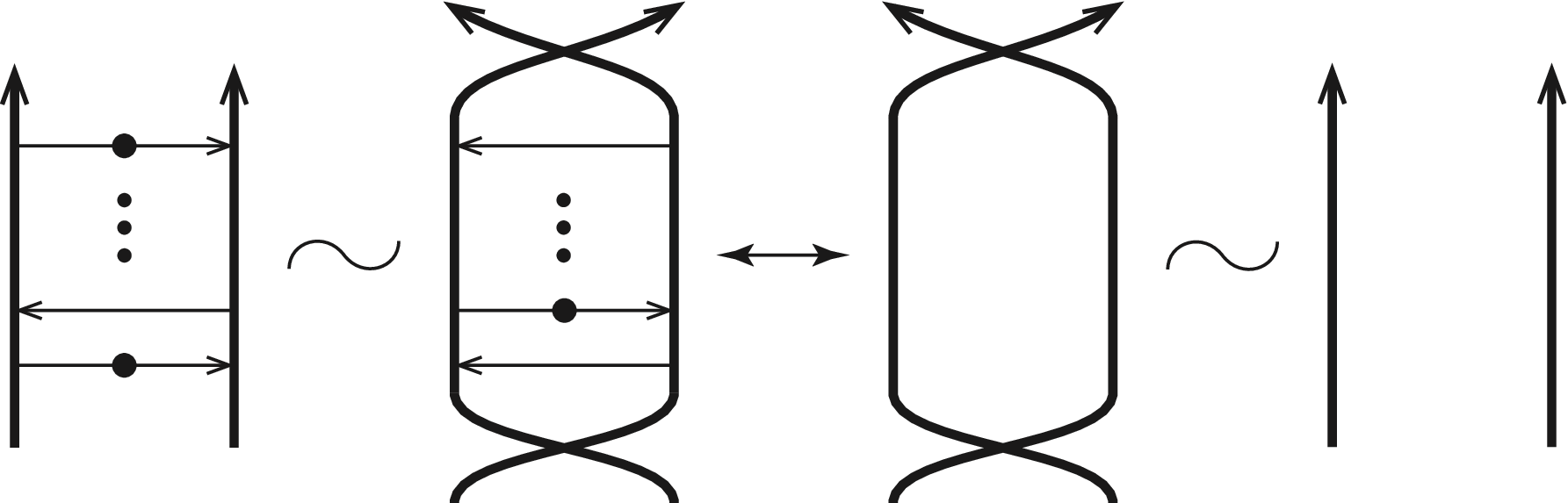}
      \put(49,45){AR}
      \put(118,45){$A(n)$}
      \put(192,45){AR}
      \put(40,20){{\footnotesize $1$}}
      \put(40,29){{\footnotesize $2$}}
      \put(40,56){{\footnotesize $n$}}
      \put(112,20){{\footnotesize $1$}}
      \put(112,29){{\footnotesize $2$}}
      \put(112,56){{\footnotesize $n$}}
    \end{overpic}
  \end{center}
  \caption{}
  \label{pf-A(n)odd}
\end{figure}

As a consequence of Lemmas~\ref{lem-UC}, \ref{lem-HvsA(n)} and \ref{lem-A(n)V(n)}, we have the following.

\begin{proposition}
\label{prop-UC}
Let $n$ be a positive integer. 
A {\rm UC}-move is realized by a sequence of the $V(n)$-move and welded Reidemeister moves. 
Hence, the $V(n)$-move is an unknotting operation for welded knots.
\end{proposition}

Here, we define the {\em $A^{n}$-move} as a local move on an arrow presentation illustrated in Figure~\ref{A^n}.

\begin{figure}[htbp]
  \begin{center}
    \begin{overpic}[width=10cm]{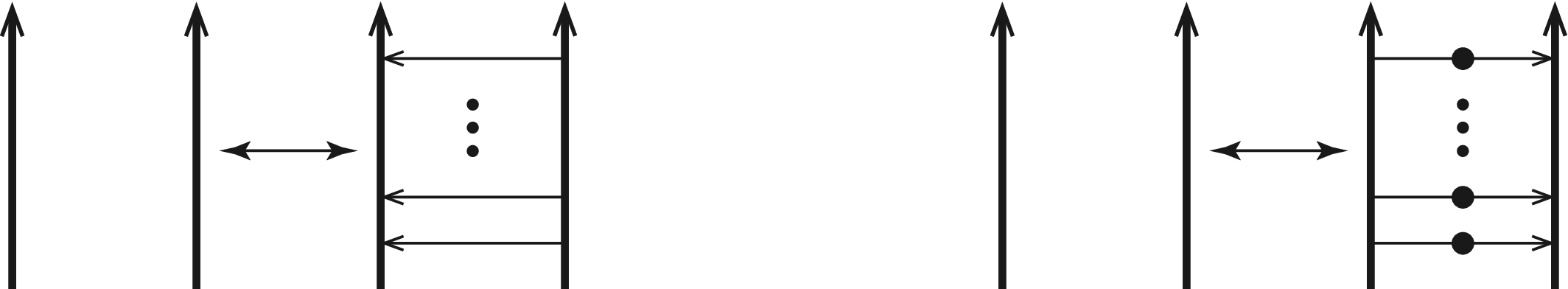}
      \put(49,29){$A^{n}$}
      \put(106,5){{\footnotesize $1$}}
      \put(106,14){{\footnotesize $2$}}
      \put(106,40){{\footnotesize $n$}}
      
      \put(227,29){$A^{n}$}
      \put(285,5){{\footnotesize $1$}}
      \put(285,14){{\footnotesize $2$}}
      \put(285,40){{\footnotesize $n$}}
    \end{overpic}
  \end{center}
  \caption{$A^{n}$-move on an arrow presentation}
  \label{A^n}
\end{figure}
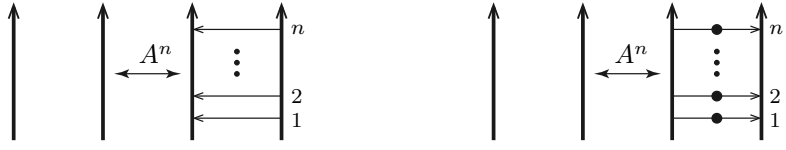

\begin{lemma}
\label{lem-realizingA^{n}}
Let $n$ be an odd number. 
An $A^{n}$-move is realized by a sequence of the $A(n)$-move and arrow moves. 
\end{lemma}

\begin{proof}
We consider the {\em head-tail reversal move} illustrated  in Figure~\ref{HTrev}, 
which is realized by a sequence of the move AR9 and $A(n)$-move.
(Figure~\ref{pf-HTrev} shows that one of the head-tail reversal moves is realized by a sequence of the move AR9 and $A(n)$-move. 
The other case is similarly shown.) 
Combining the head-tail reversal moves with an $A(n)$-move, we can realize an $A^{n}$-move. 
\end{proof}

\begin{figure}[htbp]
  \begin{center}
    \begin{overpic}[width=9cm]{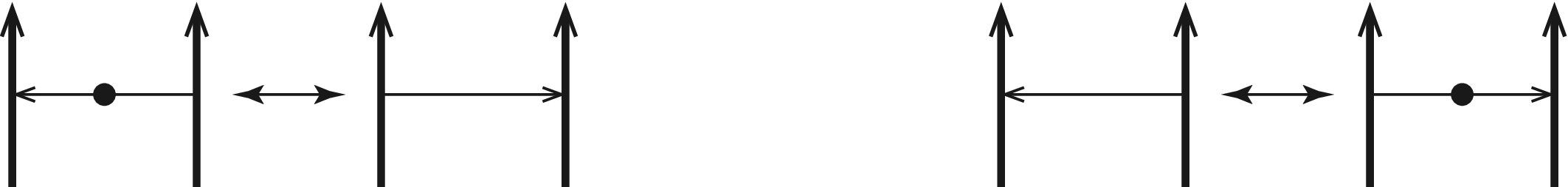}
    \end{overpic}
  \end{center}
  \caption{Head-tail reversal move}
  \label{HTrev}
\end{figure}

\begin{figure}[htbp]
  \begin{center}
    \begin{overpic}[width=8cm]{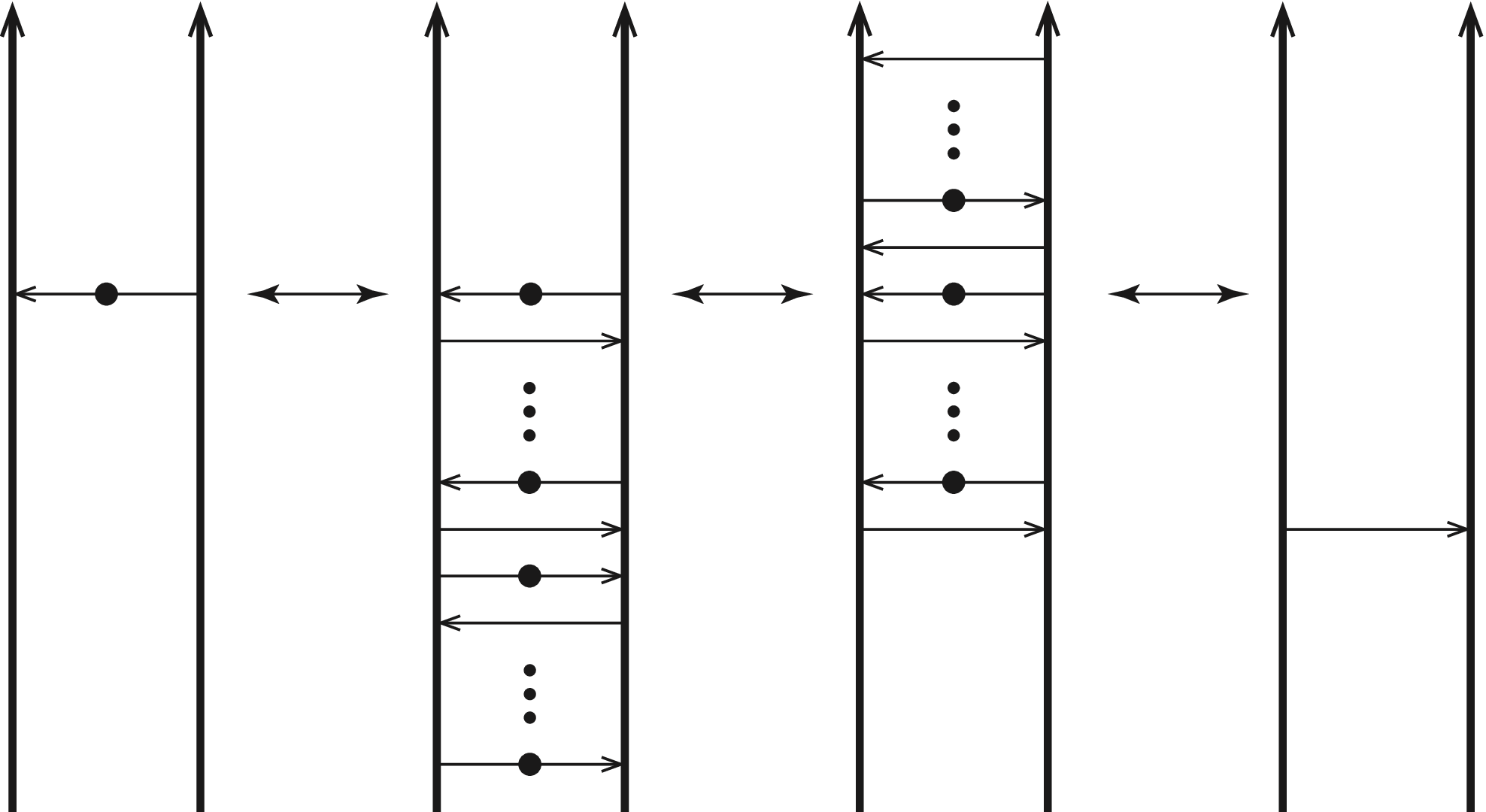}
      \put(39.5,83){AR9}
      \put(104,84){$A(n)$}
      \put(98,71){{\footnotesize $n$}}
      \put(98,48){{\footnotesize $2$}}
      \put(98,41){{\footnotesize $1$}}
      \put(98,34.5){{\footnotesize $n$}}
      \put(98,27){{\footnotesize $n-1$}}
      \put(98,5){{\footnotesize $1$}}
      \put(170.5,83){AR9}
      \put(163,113){{\footnotesize $n$}}
      \put(163,91.5){{\footnotesize $2$}}
      \put(163,84){{\footnotesize $1$}}
      \put(163,71){{\footnotesize $n$}}
      \put(163,48){{\footnotesize $2$}}
      \put(163,41){{\footnotesize $1$}}
    \end{overpic}
  \end{center}
  \caption{} 
  \label{pf-HTrev}
\end{figure}

By arguments similar to those in the proof of Lemma~\ref{lem-A(n)V(n)} we have the following.

\begin{lemma}\label{lem-A^nV^n}
An arrow presentation for a $V^{n}$-move is realized by a sequence of the $A^{n}$-move and arrow moves. 
Conversely, surgery along an $A^{n}$-move is realized by a sequence of the $V^{n}$-move and welded Reidemeister moves. 
\end{lemma}

\section{Proofs of theorems}\label{sec-proofs}
In this section, we will give the proofs of Theorems~\ref{th-even}, \ref{th-odd} and \ref{th-V^n}.

\begin{proof}[Proof of Theorem~$\ref{th-even}$]
If $n$ is even then any welded link can be deformed into some welded knot by $V(n)$-moves,  
since $V(n)$-moves can change the number of components of the welded link. 
Therefore, Theorem~\ref{th-even} follows from Proposition~\ref{prop-UC}.
\end{proof}

Let $\mathbf{1}$ be the ordered oriented $\mu$-component trivial string link diagram without crossings such that all strands are oriented upwards.
For an integer $a$, let $(\mathbf{1},H_{ij}(a))$ denote the arrow presentation of Figure~\ref{sigma}, 
that is, $H_{ij}(a)$ consists of $|a|$ horizontal {\rm w}-arrows 
whose tails (resp. heads) are attached to the $i$th (resp. $j$th)
strand of~$\mathbf{1}$ $(1\leq i<j\leq\mu)$ such that each {\rm w}-arrow has exactly one twist if $a\geq 0$, 
and no twist otherwise. 
Note that, for arrow presentations $(\mathbf{1},H_{ij}(a))$ and $(\mathbf{1},H_{kl}(a'))$, 
the stacking products $(\mathbf{1},H_{ij}(a))*(\mathbf{1},H_{kl}(a'))$ 
and $(\mathbf{1},H_{kl}(a'))*(\mathbf{1},H_{ij}(a))$ are related by  ends exchange moves and arrow moves,
hence, by $A(n)$-moves and arrow moves. 
Here, the {\em stacking product $(\mathbf{1},A)*(\mathbf{1},B)$} of arrow presentations $(\mathbf{1},A)$ and $(\mathbf{1},B)$ is the arrow presentation corresponding to the diagram $\mathbf{1}_{A}*\mathbf{1}_{B}$. 
Let $\prod_{1\leq i<j\leq \mu}(\mathbf{1},H_{ij}(a_{ij}))$ denote the stacking products of $(\mathbf{1},H_{ij}(a_{ij}))$
for integers $a_{ij}$ $(1\leq i<j\leq\mu)$. 
We remark that the ordered linking numbers $\lambda_{ij}$ and $\lambda_{ji}$ of the closure of the string link diagram $\prod_{1\leq i<j\leq \mu}\mathbf{1}_{H_{ij}(a_{ij})}$ are equal to $a_{ij}$ and $0$, respectively $(1\leq i<j\leq\mu)$.

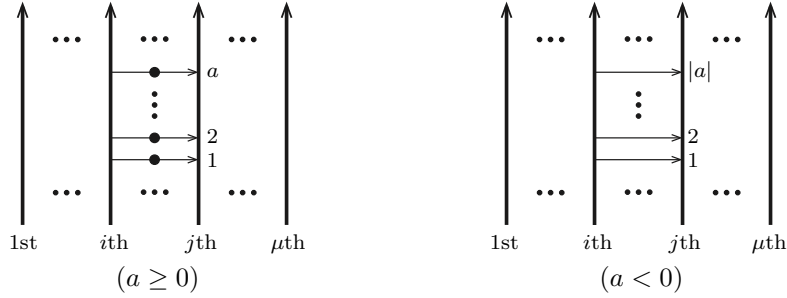
\begin{figure}[htbp]
  \begin{center}
    \begin{overpic}[width=10cm]{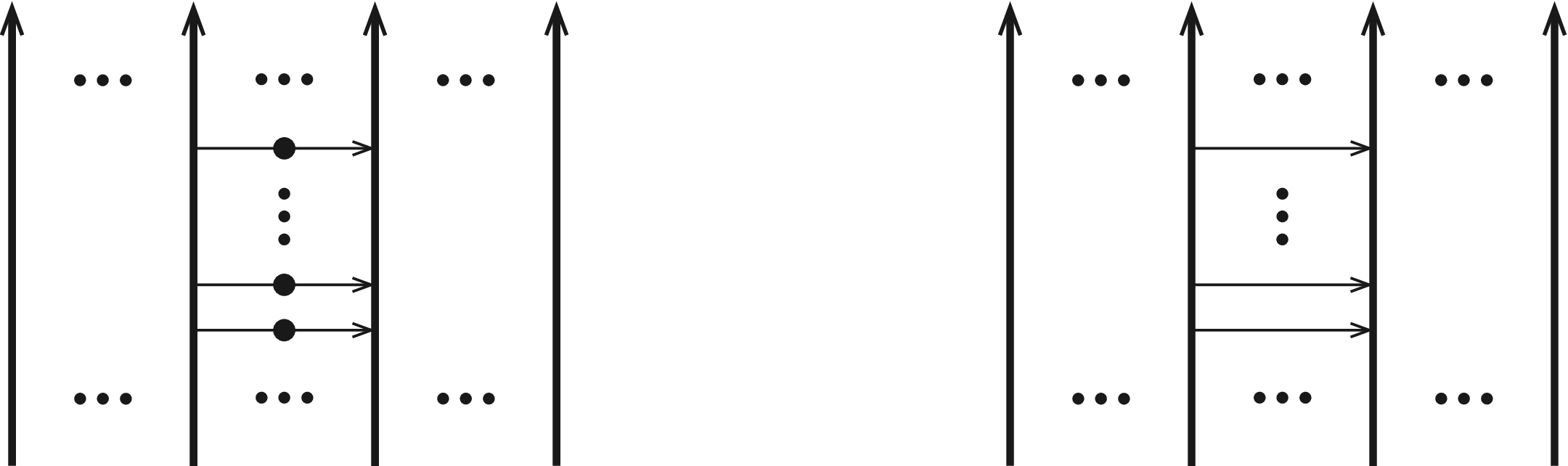}
      \put(37,-23){$(a\geq 0)$}
      \put(218,-23){$(a<0)$}
      \put(-3,-9){{\footnotesize $1$st}}
      \put(31,-9){{\footnotesize $i$th}}
      \put(63,-9){{\footnotesize $j$th}}
      \put(95,-9){{\footnotesize $\mu$th}}
      \put(71,22){{\footnotesize $1$}}
      \put(71,31){{\footnotesize $2$}}
      \put(71,56){{\footnotesize $a$}}
      
      \put(177,-9){{\footnotesize $1$st}}
      \put(212,-9){{\footnotesize $i$th}}
      \put(244,-9){{\footnotesize $j$th}}
      \put(275,-9){{\footnotesize $\mu$th}}
      \put(251,22){{\footnotesize $1$}}
      \put(251,31){{\footnotesize $2$}}
      \put(251,56){{\footnotesize $|a|$}}
    \end{overpic}
  \end{center} 
  \vspace{1.5em}
  \caption{Arrow presentation $(\mathbf{1},H_{ij}(a))$}
  \label{sigma}
\end{figure}

\begin{lemma}
\label{lem-A(n)}
Let $n$ be an odd number. 
For any $\mu$-component virtual string link diagram $D$, 
there are integers $a_{ij}$ with $0\leq a_{ij}<n$ $(1\leq i<j\leq\mu)$ such that 
an arrow presentation for $D$ and $\prod_{1\leq i<j\leq \mu}(\mathbf{1},H_{ij}(a_{ij}))$ are related by $A(n)$-moves and arrow moves.  
\end{lemma}

\begin{proof} 
Let $(\mathbf{1},\bigcup_{1\leq i,j\leq\mu}W_{ij})$ be an arrow presentation for a $\mu$-component virtual string link diagram, 
where $W_{ij}$ is a set of {\rm w}-arrows for $\mathbf{1}$ 
whose tails (resp. heads) are attached to the $i$th (resp. $j$th) strand ($1\leq i,j\leq\mu$, possibly $i=j$). 
We show that $(\mathbf{1},\bigcup_{1\leq i,j\leq\mu}W_{ij})$ can be deformed
into the desired form by $A(n)$-moves and arrow moves (including ends exchange, head-tail reversal moves and $A^{n}$-moves).
First, the ends of each {\rm w}-arrow in $W_{ii}$ $(1\leq i\leq\mu)$ can be moved into position to be removed by a single AR8. 
Hence, all {\rm w}-arrows in $W_{ii}$ are removed for any $i$.  
Next, $(\mathbf{1},\bigcup_{1\leq i\neq j\leq\mu}W_{ij})$ can be deformed into $\prod_{1\leq i<j\leq\mu}(\mathbf{1},H_{ij}(a_{ij}))$ for some integers $a_{ij}$ 
by combining head-tail reversal, ends exchange moves and AR9.
Finally, we obtain the desired form by performing $A^{n}$-moves and AR9. \end{proof}

\begin{proof}[Proof of Theorem~$\ref{th-odd}$]
It suffices to show the ``if'' part. 
Let $D$ and $D'$ be virtual link diagrams of $L$ and $L'$, respectively. 
For any virtual link diagram, there exists a virtual string link diagram 
whose closure is welded isotopic to the virtual link diagram. 
Hence, by Lemma~\ref{lem-A(n)}, two arrow presentations $(T,A)$ for $D$ and $(T',A')$ for $D'$ can be related to 
the closures of $\prod_{1\leq i<j\leq \mu}(\mathbf{1},H_{ij}(a_{ij}))$ and $\prod_{1\leq i<j\leq \mu}(\mathbf{1},H_{ij}(a'_{ij}))$, respectively, for some non-negative integers $a_{ij},a'_{ij}$ $(<n)$, by $A(n)$-moves and arrow moves. 
Then, for any $i,j$ $(1\leq i<j\leq\mu)$, we have
\[a_{ij}\equiv\lambda_{ij}(D)+\lambda_{ji}(D)\equiv\lambda_{ij}(D')+\lambda_{ji}(D')\equiv a'_{ij}
\pmod{n}.\] 
Since $0\leq a_{ij},a'_{ij}<n$ it follows that $a_{ij}=a'_{ij}$. 
Therefore, $(T, A)$ and $(T', A')$ are related by $A(n)$-moves and arrow moves.
Consequently, $D(=T_{A})$ and $D'(=T'_{A'})$ are related by $V(n)$-moves and welded Reidemeister moves. 
\end{proof}

For an integer $b$, let $(\mathbf{1},\overline{H}_{ij}(b))$ denote the arrow presentation of Figure~\ref{sigma2}, 
that is, $\overline{H}_{ij}(b)$ consists of 
$|b|$ horizontal {\rm w}-arrows 
whose heads (resp. tails) are attached to the $i$th (resp. $j$th)
strand of $\mathbf{1}$ $(1\leq i<j\leq\mu)$ such that each {\rm w}-arrow 
has no twist if $b\geq 0$, and exactly one twist otherwise.
We remark that, for integers $a_{ij}$ and $b_{ij}$, 
 the numbers $\lambda_{ij}$ and $\lambda_{ji}$ of the closure of the string link diagram $\prod_{1\leq i<j\leq\mu}(\mathbf{1}_{H_{ij}(a_{ij})}*\mathbf{1}_{\overline{H}_{ij}(b_{ij})})$ are equal to $a_{ij}$ and $b_{ij}$, respectively $(1\leq i<j\leq\mu)$.

\begin{figure}[htbp]
  \begin{center}
    \begin{overpic}[width=10cm]{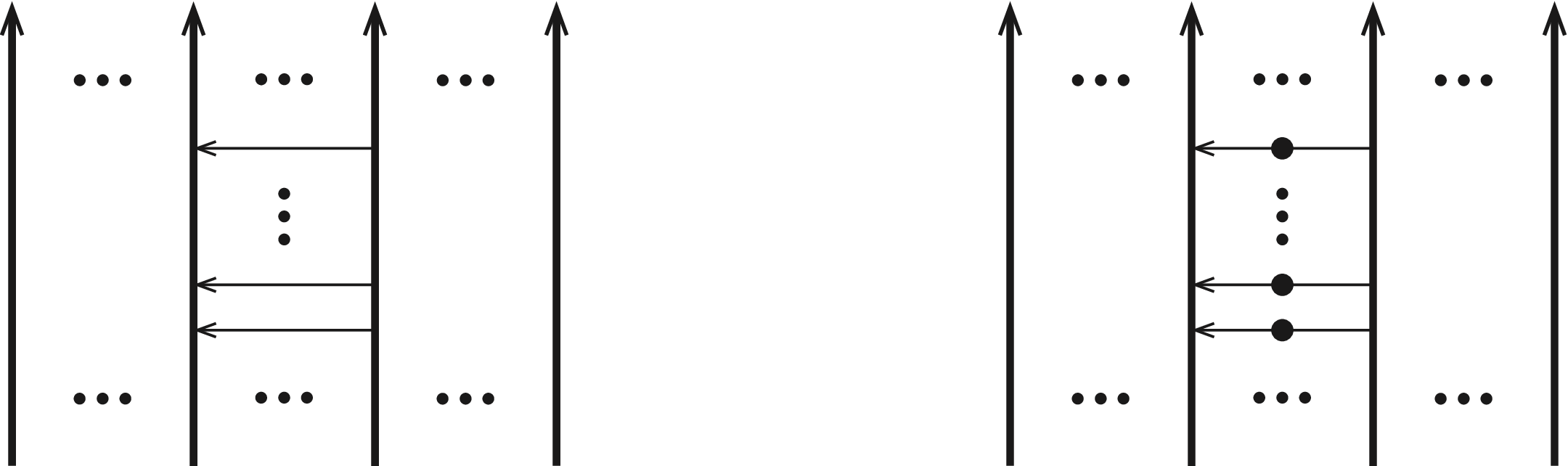}
      \put(38,-23){$(b\geq 0)$}
      \put(219,-23){$(b<0)$}
      \put(-3,-9){{\footnotesize $1$st}}
      \put(31,-9){{\footnotesize $i$th}}
      \put(63,-9){{\footnotesize $j$th}}
      \put(95,-9){{\footnotesize $\mu$th}}
      \put(71,22){{\footnotesize $1$}}
      \put(71,31){{\footnotesize $2$}}
      \put(71,56){{\footnotesize $b$}}
      
      \put(177,-9){{\footnotesize $1$st}}
      \put(212,-9){{\footnotesize $i$th}}
      \put(244,-9){{\footnotesize $j$th}}
      \put(275,-9){{\footnotesize $\mu$th}}
      \put(251,22){{\footnotesize $1$}}
      \put(251,31){{\footnotesize $2$}}
      \put(251,56){{\footnotesize $|b|$}}
    \end{overpic}
  \end{center} 
  \vspace{1.5em}
  \caption{Arrow presentation $(\mathbf{1},\overline{H}_{ij}(b))$}
  \label{sigma2}
\end{figure}
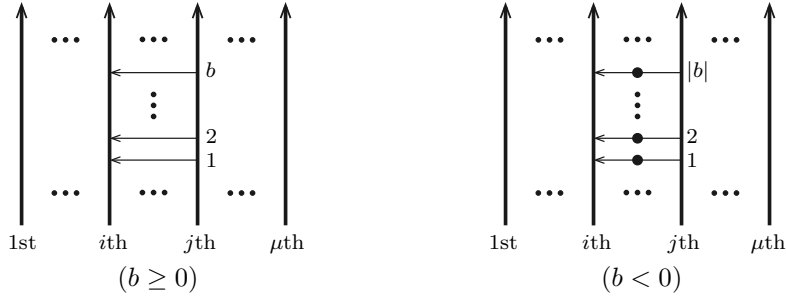

\begin{lemma}
\label{lem-A^n}
Let $n$ be a positive integer.  
For any $\mu$-component virtual string link diagram $D$, 
there are integers $a_{ij},b_{ij}$ with $0\leq a_{ij},b_{ij}<n$ $(1\leq i<j\leq\mu)$ such that 
an arrow presentation for $D$ and $\prod_{1\leq i<j\leq \mu}((\mathbf{1},H_{ij}(a_{ij}))*(\mathbf{1},\overline{H}_{ij}(b_{ij})))$ are related by $A^{n}$-moves, ends exchange moves and arrow moves. 
\end{lemma}

The proof of the lemma above can be done by a similar way to the proof of Lemma~\ref{lem-A(n)}. 
Note that we are not permitted to use the head-tail reversal move.
This is the reason why we need not only $H_{ij}(a)$ but also $\overline{H}_{ij}(b)$.

\begin{proof}[Proof of Theorem~$\ref{th-V^n}$]
It suffices to show the ``if'' part. 
Let $D$ and $D'$ be virtual link diagrams of $L$ and $L'$, respectively.  
By Lemma~\ref{lem-A^n}, two arrow presentations $(T,A)$ for $D$ and $(T',A')$ for $D'$ can be related to 
the closures of $\prod_{1\leq i<j\leq \mu}((\mathbf{1},H_{ij}(a_{ij}))*(\mathbf{1},\overline{H}_{ij}(b_{ij})))$ and $\prod_{1\leq i<j\leq \mu}((\mathbf{1},H_{ij}(a'_{ij}))*(\mathbf{1},\overline{H}_{ij}(b'_{ij})))$, respectively, for some non-negative integers $a_{ij},b_{ji},a'_{ij},b'_{ji}$ $(<n)$ by $A^{n}$-moves, ends exchange moves and arrow moves.  
Then, for any $i,j$ $(1\leq i<j\leq\mu)$, we have
\[a_{ij}\equiv\lambda_{ij}(D)\equiv\lambda_{ij}(D')\equiv a'_{ij}\pmod{n}\] 
and 
\[b_{ij}\equiv\lambda_{ji}(D)\equiv\lambda_{ji}(D')\equiv b'_{ij}
\pmod{n}.\] 
Since $0\leq a_{ij},b_{ij},a'_{ij}b'_{ij}<n$ it follows that $a_{ij}=a'_{ij}$ and $b_{ij}=b'_{ij}$.  
Therefore, $(T, A)$ and $(T', A')$ are related by $A^n$-moves, ends exchange moves and arrow moves.
Lemmas~\ref{lem-ends},~\ref{lem-UC} and~\ref{lem-A^nV^n} imply that 
 $D(=T_{A})$ and $D'(=T'_{A'})$ are related by $V^{n}$-moves, UC-moves and welded Reidemeister moves.  
\end{proof}

\section{$V^{n}$-moves and UC-moves}\label{sec-UC}
As mentioned in Section 1, 
the $V^{n}$-move is not an unknotting operation except for $n=1$.
To prove this, we use the Alexander polynomials 
which are obtained from the group of a welded link by using the Fox free derivative. 
Here, the {\em group} of a virtual link diagram is known to be a welded link invariant~\cite[Section 4]{Kauffman}, 
and hence 
(the elementary ideals in the sense of~\cite{CF} and) 
the Alexander polynomials are naturally extended to welded link invariants. 
By a similar way to the proof of Theorem~1 in~\cite{Kinoshita}, 
we can show the following.

\begin{proposition}
\label{prop-Alex}
Let $n$ be a positive integer. 
If two welded links $L$ and $L'$ are $V^{n}$-equivalent, 
then for a non-negative integer $k$ and 
for the $k$th elementary ideals $E^{k}_{L}(t)$ and $E^{k}_{L'}(t)$ of $L$ and $L'$, respectively, we have 
\[
E^{k}_{L}(t)\equiv E^{k}_{L'}(t)\mod{I(1-t^{n})},  
\]
where $I(1-t^{n})$ is the ideal generated by $1-t^{n}$ in $\mathbb{Z}[t^{\pm 1}]$. 
In particular, 
for the $($$1$-variable$)$ 
$k$th Alexander polynomials 
$\Delta_{L}^{k}(t)$ and $\Delta_{L'}^{k}(t)$ of $L$ and $L'$, respectively, we have 
\[
\Delta_{L}^{k}(t)\equiv\e t^{r}\Delta_{L'}^{k}(t)\mod{I(1-t^{n})} 
\]
for some $\e\in\{\pm1\}$ and $r\in\mathbb{Z}$.  
\end{proposition}

\begin{proof}
Let $D$ and $ D'$ be virtual link diagrams of $L$ and $L'$, respectively. 
It suffices to show that if $D$ and $D'$ are related by a single $V^{n}$-move then for their properly chosen Alexander matrices $A_{D}(t)$ and $A_{D'}(t)$, 
\[A_{D}(t)\equiv A_{D'}(t)\mod{I(1-t^{n})}.\]

Suppose that $D'$ is obtained from $D$ by a single R1 and a single $V^{n}$-move, 
and put labels $x_{1},x_{2}$ and $x_{3}$ on arcs of $D$ and $D'$ as illustrated  in Figure~\ref{pf-Alex} and labels $x_{4},\ldots,x_{l}$ on the other arcs outside the figure.

\begin{figure}[htbp]
  \begin{center}
    \begin{overpic}[width=11cm]{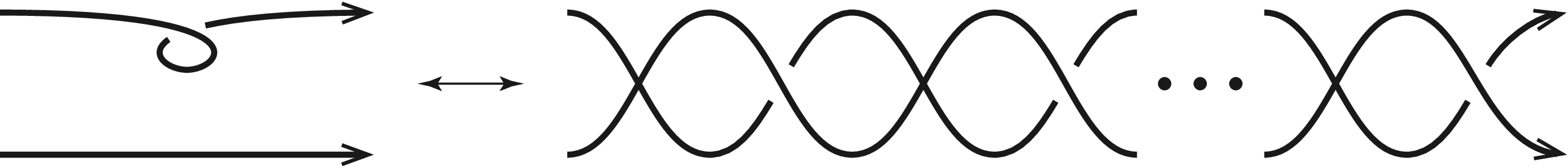}
      \put(33,-18){$D$}
      \put(200,-18){$D'$}
      \put(81,20.5){R1,$V^{n}$}
      \put(154,2){{\footnotesize $1$}}
      \put(210.5,2){{\footnotesize $2$}}
      \put(293,2){{\footnotesize $n$}}
      \put(-5,35){$x_{1}$}
      \put(72,35){$x_{2}$}
      \put(-5,-5.5){$x_{3}$}
      \put(72,-5.5){$x_{3}$}
      \put(107,35){$x_{1}$}
      \put(107,-5.5){$x_{3}$}
      \put(311,35){$x_{2}$}
      \put(311,-5.5){$x_{3}$}
    \end{overpic}
  \end{center}
  \vspace{1em}
  \caption{}
  \label{pf-Alex}
\end{figure}

Then, we obtain group presentations of the groups $G(D)$ and $G(D')$ of $D$ and $D'$, respectively, as follows: 
\[
\begin{array}{ll}
G(D)=\langle x_{1},x_{2},x_{3},x_{4},\ldots,x_{l}\mid x_{1}x_{2}^{-1},\{r_{i}\}\rangle,  \\
G(D')=\langle x_{1},x_{2},x_{3},x_{4},\ldots,x_{l}\mid x_{1}x_{3}^{n}x_{2}^{-1}x_{3}^{-n},\{r_{i}\}\rangle, 
\end{array}
\]
where $\{r_{i}\}$ is the set of relations corresponding to the other crossings. 
Using the Fox free derivative~\cite{CF}, we have the Alexander matrices $A_{D}(t)$ and $A_{D'}(t)$ of $D$ and $D'$, respectively, as follows: 
\[
A_{D}(t)=\begin{pmatrix}
1 &-1 &0 &0 &\cdots &0 \\ \cdashline{1-6}[3pt/3pt]
\multicolumn{6}{c}{} \\
\multicolumn{6}{c:}{\raisebox{0pt}[0pt][0pt]{$\mathfrak{a}\gamma\left(\cfrac{r_{i}}{x_{j}}\right)$}} \\
\multicolumn{6}{c}{} \\
\end{pmatrix}, 
A_{D'}(t)=\begin{pmatrix}
1 &-t^{n} &t^{n}-1 &0 &\cdots &0 \\ \cdashline{1-6}[3pt/3pt] 
\multicolumn{6}{c}{} \\
\multicolumn{6}{c:}{\raisebox{0pt}[0pt][0pt]{$\mathfrak{a}\gamma\left(\cfrac{r_{i}}{x_{j}}\right)$}} \\
\multicolumn{6}{c}{} \\
\end{pmatrix}.
\]
Therefore, $A_{D}(t)-A_{D'}(t)$ is a zero matrix modulo ${I(1-t^{n})}$.  
\end{proof}

Now we have the following.

\begin{proposition}
\label{prop-trefoil}
The $V^{n}$-move is not an unknotting operation on welded knots for $n\geq 2$.
\end{proposition}

\begin{proof}
We show that the trefoil knot is not $V^{n}$-equivalent to the unknot for $n\geq 2$.
The first Alexander polynomial 
of the trefoil knot is $1-t+t^{2}$, and that of the unknot is $1$. 
It suffices to show that $1-t+t^{2}-\e t^{r}\not\in I(1-t^{n})$ for any $n\geq 2$ by Proposition~\ref{prop-Alex} $(\e\in\{\pm 1\},r\in\mathbb{Z})$.

Suppose that $n\geq 2$. 
We define a map
$f_{n}:\mathbb{Z}[t^{\pm 1}]\rightarrow \mathbb{Z}$ by $f_{n}(\sum_{i}a_{i}t^{i})=\sum_{i\equiv 0,2\pmod{n}}a_{i}$, where $a_{i}\in\mathbb{Z}$. 
Since 
$f_{n}(\de t^{s}(1-t^{n}))=0$ $(\de\in\{\pm 1\}, s\in\mathbb{Z})$ it follows that 
$f(b)=0$ for any element $b\in I(1-t^{n})$. 
On the other hand, $f_{n}(1-t+t^{2}-\e t^{r})$ is not equal to $0$ for any $n\geq 2$. 
This completes the proof. 
\end{proof}

The proposition above immediately implies the following corollary.

\begin{corollary}
\label{cor-UC}
A {\rm UC}-move is realized by a sequence of the $V^{n}$-move and welded Reidemeister moves if and only if $n=1$.
\end{corollary}

\section{Unoriented $V^{n}$-moves and associated core groups}
In this section, we will discuss relations among {\em unoriented} $V^{n}$-moves, associated core groups and the multiplexing of crossings.

For an unoriented classical link diagram $D$, 
the {\em associated core group $\Pi^{(2)}_{D}$} is defined as follows. 
Each arc of $D$ yields a generator, and each classical crossing gives a relation $yx^{-1}yz^{-1}$, where $x$ and $z$ correspond to the underpasses and $y$ corresponds to the overpass at the crossing. 
This group $\Pi^{(2)}_{D}$ is known as a classical link invariant~\cite{J,Kelly,FR,W}.

\begin{remark}
\label{rem-MWada}
Let $L$ be an unoriented classical link in the $3$-sphere and $D$ a classical diagram of $L$. 
M.~Wada~\cite{W} proved that $\Pi_{D}^{(2)}$ is isomorphic to the free product of 
the fundamental group of the double branched cover $M_{L}^{(2)}$ of the $3$-sphere branched along $L$ and the infinite cyclic group $\mathbb{Z}$. 
That is, $\Pi_{D}^{(2)}\cong\pi_{1}\left(M_{L}^{(2)}\right)*\mathbb{Z}$. 
\end{remark}

We similarly define the associated core group $\Pi^{(2)}_{D}$ of an unoriented {\em virtual} link diagram $D$ by generators and relations as described above. 
(Note that virtual crossings do not produce any generator or relation.) 
It is not hard to see that $\Pi^{(2)}_{D}$ is preserved by welded Reidemeister moves, 
and hence we define the {\em associated core group $\Pi^{(2)}_{L}$} of an unoriented welded link $L$ to be $\Pi^{(2)}_{D}$ of a diagram $D$ of $L$. 
Moreover we have the following.

\begin{proposition}
\label{prop-V2}
If $n$ is even, then $\Pi^{(2)}_{L}$ is preserved by unoriented $V^{n}$-moves. 
\end{proposition}

\begin{proof} 
$\Pi^{(2)}_{L}$ is preserved by unoriented $V^{2}$-moves as illustrated  in Figure~\ref{V2}, 
and furthermore, an unoriented $V^{n}$-move is realized by unoriented $V^{2}$-moves for any even number $n$. 
\end{proof}

\begin{figure}[htbp]
  \begin{center}
    \hspace{-7em}
    \begin{overpic}[width=8cm]{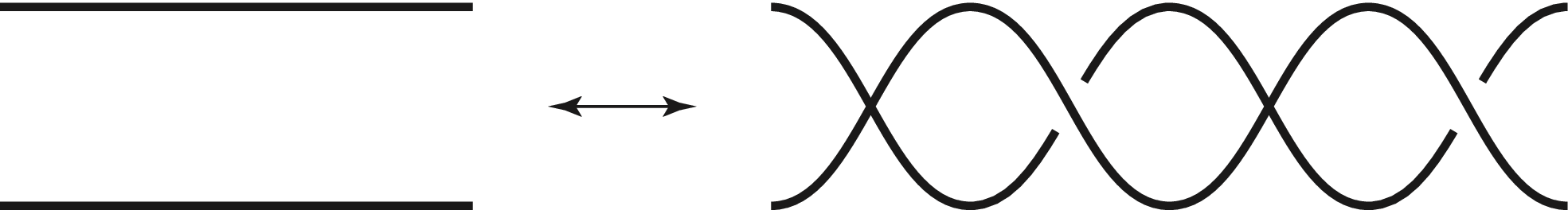}
      \put(-7,28){$x$}
      \put(-8,-1){$y$}
      \put(86.5,20){$V^{2}$} 
      \put(104,28){$x$} 
      \put(104,-1){$y$}
      \put(138.5,34){$y$} 
      \put(138.5,-8){$x$}
      \put(158,34){$yx^{-1}y$}
      \put(187,-8){$yx^{-1}y$}
      \put(232,28){$y(y^{-1}xy^{-1})y=x$}
      \put(232,-1){$y$} 
    \end{overpic}
  \end{center}
  \caption{}
  \label{V2}
\end{figure}

There are welded knots whose associated core groups are nontrivial,  
for example, all knots having nontrivial Fox colorings; see~\cite[Proposition~4.5]{P}.
Therefore, the proposition above gives an alternative proof for that 
the $V^{n}$-move is not an unknotting operation on welded knots for any even number $n$.

In~\cite{MWY17}, the authors introduced the {\em multiplexing} of crossings for an unoriented virtual link diagram, 
which yields a new unoriented virtual link diagram. 
Let $(m_{1},\ldots,m_{\mu})$ be a $\mu$-tuple of integers, and 
let $D=\bigcup_{i=1}^{\mu}D_{i}$ be an unoriented $\mu$-component virtual link diagram. 
For a classical crossing of $D$ whose overpass belongs to $D_{j}$, 
we define the {\em multiplexing} of the crossing associated with $m_{j}$ as a local change illustrated  in Figure~\ref{multiplexing}. 
When $m_{j}=0$, the multiplexing of the crossing is the the crossing virtualization of it. 
The number of classical crossings that appear in the multiplexing of the crossing is the absolute value of $m_{j}$. 
Let $D(m_{1},\ldots,m_{\mu})$ denote the virtual link diagram obtained from $D$ by the multiplexing of all classical crossings of $D$ associated with $(m_{1},\ldots,m_{\mu})$. 
For welded isotopic virtual link diagrams $D$ and $D'$, 
$D(m_{1},\ldots,m_{\mu})$ and $D'(m_{1},\ldots,m_{\mu})$ are also welded isotopic for any $(m_{1},\ldots,m_{\mu})\in\mathbb{Z}^{\mu}$~\cite[Theorem~2.1]{MWY17}.
For an unoriented $\mu$-component welded link $L$, 
we define $L(m_{1},\ldots,m_{\mu})$ to be $D(m_{1},\ldots,m_{\mu})$ of a diagram $D$ of $L$.

It is not hard to see that 
$L(m_{1},\ldots,m_{\mu})$ can be deformed into $L(0,\ldots,0)$ by unoriented $V^{2}$-moves for any $\mu$-tuple $(m_{1},\ldots,m_{\mu})$ of even numbers. 
Since $L(0,\ldots,0)$ is trivial, we have the following.

\begin{figure}[htbp]
  \begin{center}
    \begin{overpic}[width=10cm]{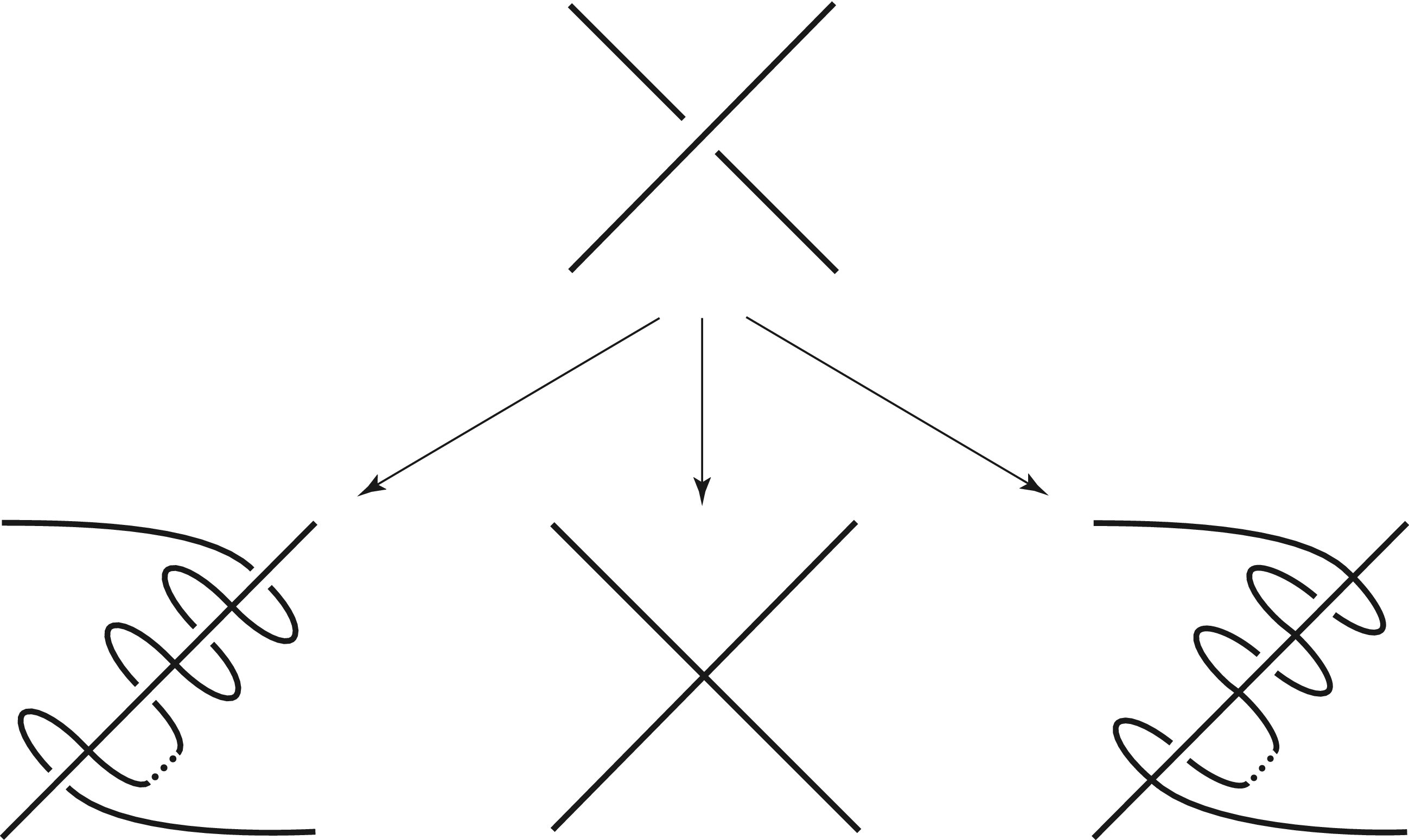}
      \put(14,-15){$m_{j}>0$}
      \put(125,-15){$m_{j}=0$}
      \put(235,-16){$m_{j}<0$}
      \put(160,150){$\leftarrow$ $D_{j}$}
      \put(51,57){{\scriptsize $1$}}
      \put(-4,11){{\scriptsize $m_{j}$}}
      \put(265,51){{\scriptsize $1$}}
      \put(227,27){{\scriptsize $|m_{j}|$}}
    \end{overpic}
  \end{center}
  \vspace{1em}
  \caption{Multiplexing of a crossing}
  \label{multiplexing}
\end{figure}
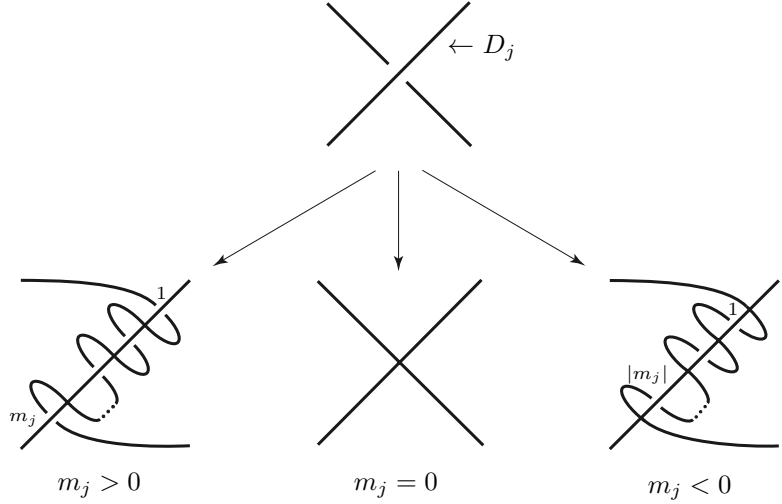

\begin{proposition}
\label{th-V2}
Let $(m_{1},\ldots,m_{\mu})$ be a $\mu$-tuple of even numbers.  
For any unoriented $\mu$-component welded link $L$, 
$L(m_{1},\ldots,m_{\mu})$ is deformed into the $\mu$-component trivial link by unoriented $V^{2}$-moves.  
\end{proposition}

In~\cite[Theorem~3.2]{MWY17}, the authors proved that unoriented 
classical knots $K$ and $K'$ are equivalent up to mirror image if and only if $K(m)$ and $K'(m)$ are welded isotopic up to mirror image for any fixed non-zero integer $m$. 
Hence, it follows that 
if a classical knot $K$ is nontrivial then $K(m)$ is also nontrivial. 
By Propositions~\ref{prop-V2} and~\ref{th-V2}, if $m$ is even then $\Pi^{(2)}_{K(m)}$ is isomorphic to that of the unknot, that is, $\Pi^{(2)}_{K(m)}\cong\mathbb{Z}$.  
Therefore, we have the following theorem 
although by Remark~\ref{rem-MWada} the associated core groups seem to be very strong invariants.

\begin{theorem}
\label{prop-infinitely}
Let $m$ $(\neq 0)$ be an even number. 
For any nontrivial unoriented welded knot $K$, 
$K(m)$ is nontrivial and $\Pi^{(2)}_{K(m)}\cong\mathbb{Z}$. 
\end{theorem}

\section{$\overline{V}(n)$-moves and $\overline{V}^{n}$-moves}

\subsection{$\overline{V}(n)$-moves}
When $n$ is odd, one may consider the $V(n)$-move involving two strands being oriented antiparallel. 
We call such a move the {\em $\overline{V}(n)$-move}. 
In this subsection, we will show that $V(n)$- and $\overline{V}(n)$-moves are equivalent. 

For an odd number $n$, we define the {\em $\overline{A}(n)$-move} as a local move on an arrow presentation illustrated in Figure~\ref{A(n)opposite}. 
By arguments similar to those in the proof of Lemma~\ref{lem-A(n)V(n)}, we have the following.

\begin{figure}[htbp]
  \begin{center}
    \begin{overpic}[width=10cm]{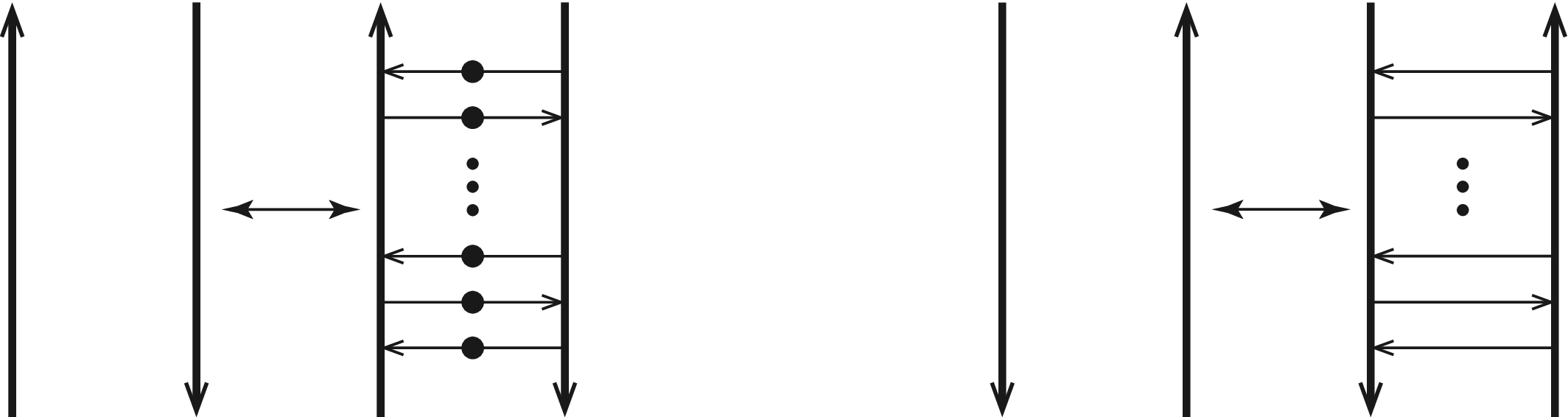}
      \put(44,43){$\overline{A}(n)$}
      \put(106,10){{\footnotesize $1$}}
      \put(106,19){{\footnotesize $2$}}
      \put(106,61){{\footnotesize $n$}}
      
      \put(222,43){$\overline{A}(n)$}
      \put(284.5,10){{\footnotesize $1$}}
      \put(284.5,19){{\footnotesize $2$}}
      \put(284.5,61){{\footnotesize $n$}}
    \end{overpic}
  \end{center}
  \caption{$\overline{A}(n)$-move on an arrow presentation}
  \label{A(n)opposite}
\end{figure}
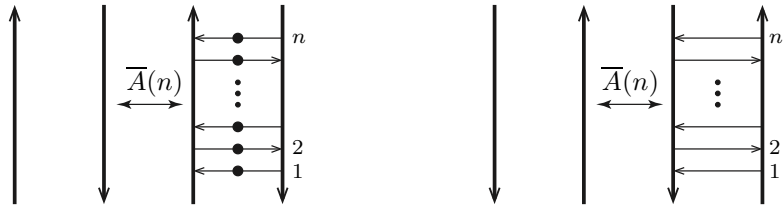

\begin{lemma}
Let $n$ be an odd number. 
An arrow presentation for a $\overline{V}(n)$-move is realized by a sequence of the $\overline{A}(n)$-move and arrow moves. 
Conversely, surgery along an $\overline{A}(n)$-move is realized by a sequence of the $\overline{V}(n)$-move and welded Reidemeister moves. 
\end{lemma}

By deformations similar to those in Figure~\ref{pf-Hexch}, 
we have the following.

\begin{lemma}\label{lem-Hdash}
An {\rm H}-move is realized by a sequence of the ${\rm H}'$-move illustrated in Figure~$\ref{Hdash}$ and arrow moves. 
\end{lemma}

\begin{figure}[htbp]
  \begin{center}
    \begin{overpic}[width=5cm]{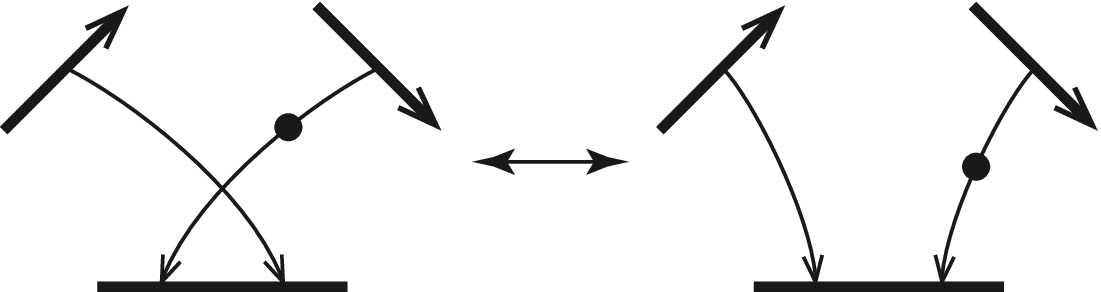}
      \put(67,20){${\rm H}'$}
    \end{overpic}
  \end{center}
  \caption{${\rm H}'$-move}
  \label{Hdash}
\end{figure}

Here, we consider two {\em allowable moves ${\rm AR11}'$ and ${\rm AR12}'$} illustrated in Figure~\ref{AR1112dash}, 
each of which is realized by a sequence of arrow moves similar to that in Figure~\ref{pf-AR112}.  
Using the moves ${\rm AR11}'$ and ${\rm AR12}'$, we have the following.

\begin{figure}[htbp]
  \begin{center}
    \begin{overpic}[width=11cm]{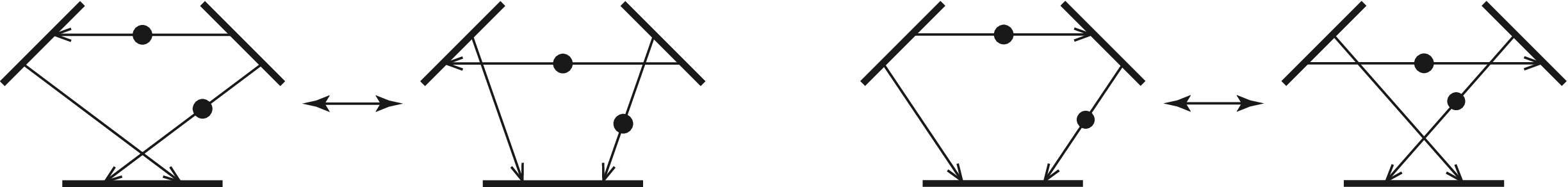}
      \put(58,20.5){${\rm AR11}'$} 
      \put(230,20.5){${\rm AR12}'$}
    \end{overpic}
  \end{center}
  \caption{Allowable moves ${\rm AR11}'$ and ${\rm AR12}'$}
  \label{AR1112dash}
\end{figure}

\begin{lemma}\label{lem-A(n)V(n)-bar}
Let $n$ be an odd number. 
An ends exchange move is realized by a sequence of the $\overline{A}(n)$-move and arrow moves. 
\end{lemma}

\begin{proof}
By Lemmas~\ref{lem-ends} and ~\ref{lem-Hdash}, 
it is enough to show that an ${\rm H}'$-move is realized by a sequence of the $\overline{A}(n)$-move and arrow moves.
Figure~\ref{pf-A_n-bar} indicates the proof. 
While the proof describes only when $n=3$ in the figure, 
it is essentially the same in all cases. 
\end{proof}

\begin{figure}[htbp]
  \begin{center}
    \begin{overpic}[width=12cm]{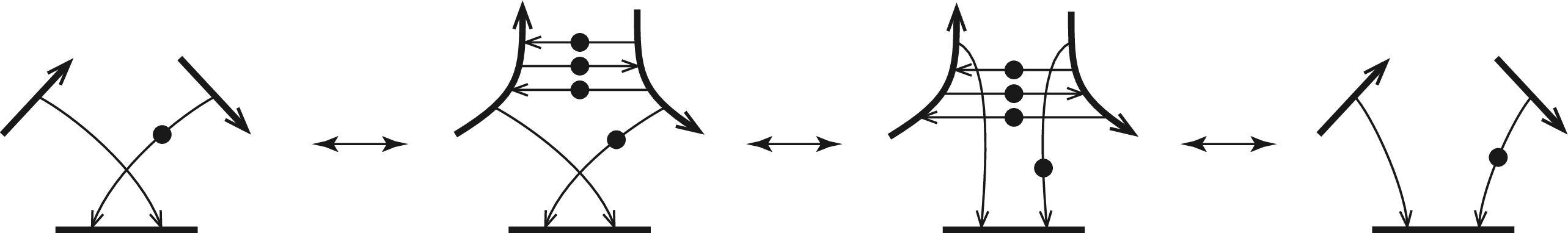}
      \put(68,25){$\overline{A}(n)$} 
      \put(152.5,25){${\rm AR11}', 12'$}
      \put(257,25){$\overline{A}(n)$} 
    \end{overpic}
  \end{center}
  \caption{}
  \label{pf-A_n-bar}
\end{figure}

Now we can show the following.

\begin{proposition}
Let $n$ be an odd number. 
$V(n)$- and $\overline{V}(n)$-moves are equivalent. 
\end{proposition}

\begin{proof}
By Lemmas~\ref{lem-A(n)V(n)} and \ref{lem-A(n)V(n)-bar}, 
it is enough to show that $A(n)$- and $\overline{A}(n)$-moves are equivalent. 
Figure~\ref{pf-A(n)} shows that an $\overline{A}(n)$-move is realized by a sequence of the $A(n)$-move and arrow moves, 
while the figure describes only when $n=3$. 
(Note that we can use ends exchange moves by Lemma~\ref{lem-HvsA(n)}.) 
Conversely, by deformations similar to those in the figure, it is not hard to see that an $A(n)$-move is realized by a sequence of the $\overline{A}(n)$-move and arrow moves. 
This completes the proof. 
\end{proof}

\begin{figure}[htbp]
  \begin{center}
    \begin{overpic}[width=12cm]{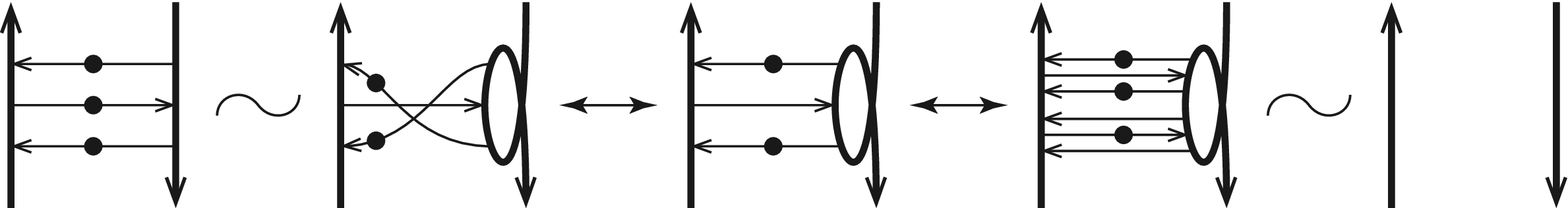} 
      \put(49,27){AR}
      \put(125,27){{\scriptsize ends}}
      \put(117,14){{\scriptsize exchange}}
      \put(199,27){$A(n)$}
      \put(277.5,27){AR}
    \end{overpic}
  \end{center}
  \caption{}
  \label{pf-A(n)}
\end{figure}

\subsection{$\overline{V}^{n}$-moves}

For a positive integer~$n$, we define the {\em $\overline{V}^{n}$-move} as the $V^{n}$-move involving two strands being oriented antiparallel. 
Also, we define the {\em $A^{n}$-move} as a local move on an arrow presentation illustrated in Figure~\ref{A^n-bar}. 
By arguments similar to those in the proof of Lemma~\ref{lem-A(n)V(n)}, we have the following.

\begin{figure}[htbp]
  \begin{center}
    \begin{overpic}[width=10cm]{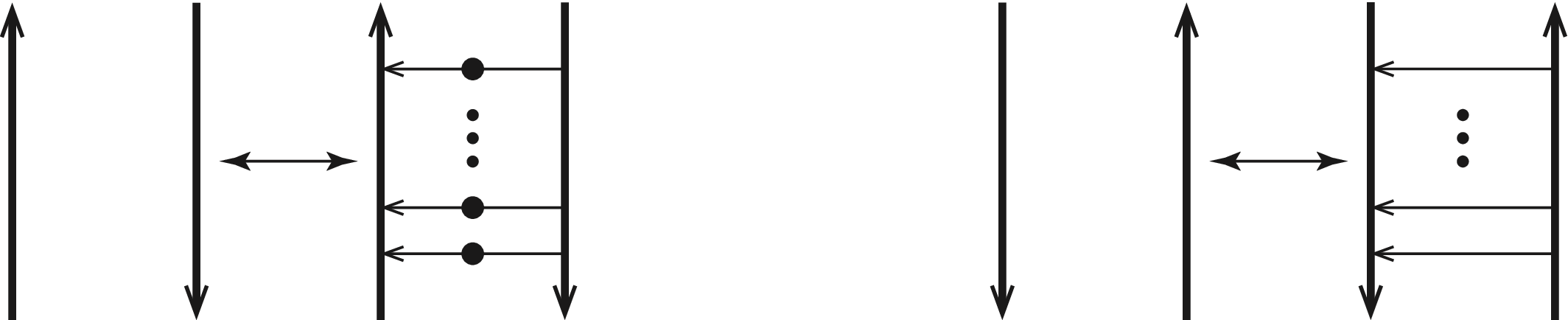}
      \put(48,32){$\overline{A}^{n}$}
      \put(105,9){{\footnotesize $1$}}
      \put(105,18){{\footnotesize $2$}}
      \put(105,43){{\footnotesize $n$}}
      
      \put(227,32){$\overline{A}^{n}$}
      \put(284.5,9){{\footnotesize $1$}}
      \put(284.5,18){{\footnotesize $2$}}
      \put(284.5,43){{\footnotesize $n$}}
    \end{overpic}
  \end{center}
  \caption{$\overline{A}^{n}$-move on an arrow presentation}
  \label{A^n-bar}
\end{figure}
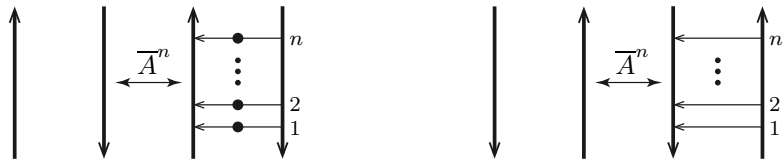

\begin{lemma}\label{lem-A^nV^n-bar}
Let $n$ be a positive integer. 
An arrow presentation for a $\overline{V}^{n}$-move is realized by a sequence of the $\overline{A}^{n}$-move and arrow moves. 
Conversely, surgery along an $\overline{A}^{n}$-move is realized by a sequence of the $\overline{V}^{n}$-move and welded Reidemeister moves. 
\end{lemma}

It is not hard to see that $A^{n}$- and $\overline{A}^{n}$-moves are equivalent. 
This together with Lemmas~\ref{lem-A^nV^n} and~\ref{lem-A^nV^n-bar} implies the following.

\begin{proposition}
Let $n$ be a positive integer. 
$V^{n}$- and $\overline{V}^{n}$-moves are equivalent. 
\end{proposition}



\end{document}